\documentclass[a4paper]{article}

\usepackage[T1]{fontenc}
\usepackage[latin9]{inputenc}
\usepackage{geometry}
\usepackage{amsmath}
\usepackage{amsthm}
\usepackage{amssymb}
\usepackage{color}
\usepackage[unicode=true,pdfusetitle,
 bookmarks=true,bookmarksnumbered=false,bookmarksopen=false,
 breaklinks=false,pdfborder={0 0 0},pdfborderstyle={},backref=false,colorlinks=false]
 {hyperref}

\makeatletter

\usepackage{appendix}

\usepackage{bm}

\usepackage[nameinlink,capitalise,noabbrev]{cleveref}
\crefname{appsec}{Appendix}{Appendices}

\hypersetup{%
    bookmarksnumbered, bookmarksopen=true, bookmarksopenlevel=1,%
}

\theoremstyle{plain}
\newtheorem{thm}{Theorem}[section]
\crefname{thm}{Theorem}{Theorems}
\theoremstyle{plain}
\newtheorem{lem}[thm]{Lemma}
\crefname{lem}{Lemma}{Lemmas}
\theoremstyle{plain}
\newtheorem{cor}[thm]{Corollary}
\theoremstyle{plain}
\newtheorem*{claim*}{Claim}
\crefname{claim}{Claim}{Claims}
\theoremstyle{definition}
\newtheorem{defn}[thm]{Definition}
\theoremstyle{plain}
\newtheorem{conjecture}[thm]{Conjecture}
\crefname{conjecture}{Conjecture}{Conjectures}
\theoremstyle{plain}

\crefname{prop}{Proposition}{Propositions}
\theoremstyle{definition}

\theoremstyle{definition}
\newtheorem{rem}[thm]{Remark}
\theoremstyle{plain}

\crefformat{equation}{#2(#1)#3}
\crefname{appsec}{Appendix}{Appendices}

\date{}

\let\originalleft\left
\let\originalright\right
\renewcommand{\left}{\mathopen{}\mathclose\bgroup\originalleft}
\renewcommand{\right}{\aftergroup\egroup\originalright}
\usepackage{verbatim}

\makeatletter
\renewcommand*{\UrlTildeSpecial}{%
  \do\~{%
    \mbox{%
      \fontfamily{ptm}\selectfont
      \textasciitilde
    }%
  }%
}%
\let\Url@force@Tilde\UrlTildeSpecial
\makeatother

\usepackage{tikz}
\tikzstyle{vertex}=[circle,draw=black,fill=black,inner sep=0,minimum size=0.2cm,text=white,font=\footnotesize]
\tikzset{every loop/.style={min distance=50,in=50,out=130,looseness=7}}

\usepackage[labelfont=bf,labelsep=period]{caption}

\let\OLDthebibliography\thebibliography
\renewcommand\thebibliography[1]{
  \OLDthebibliography{#1}
  \setlength{\parskip}{0pt}
  \setlength{\itemsep}{3pt plus 0.3ex}
}

\makeatother

\allowdisplaybreaks
\begin{document}
\title{
Dirac-type theorems in random hypergraphs
}
\author{Asaf Ferber\thanks{Department of Mathematics, University of California, Irvine.
Email: \href{mailto:asaff@uci.edu} {\nolinkurl{asaff@uci.edu}}.
Research supported in part by an NSF grant DMS-1954395.}\and Matthew Kwan\thanks{Department of Mathematics, Stanford University, Stanford, CA 94305.
Email: \href{mattkwan@stanford.edu}{\nolinkurl{mattkwan@stanford.edu}}.
Research supported in part by SNSF project 178493 and NSF award DMS-1953990.}}

\maketitle
\global\long\def\E{\mathbb{E}}%
\global\long\def\Var{\operatorname{Var}}%
\global\long\def\floor#1{\left\lfloor #1\right\rfloor }%
\global\long\def\ceil#1{\left\lceil #1\right\rceil }%

\begin{abstract}
For positive integers $d<k$ and $n$ divisible by $k$, let $m_{d}\left(k,n\right)$
be the minimum $d$-degree ensuring the existence of a perfect matching
in a $k$-uniform hypergraph. In the graph case (where $k=2$), a
classical theorem of Dirac says that $m_{1}\left(2,n\right)=\ceil{n/2}$.
However, in general, our understanding of the values of $m_{d}\left(k,n\right)$
is still very limited, and it is an active topic of research to determine
or approximate these values. In this paper we prove a ``transference''
theorem for Dirac-type results relative to random hypergraphs. Specifically,
for any $d< k$, any $\varepsilon>0$ and any ``not too small''
$p$, we prove that a random $k$-uniform hypergraph $G$ with $n$ vertices and edge probability $p$
typically has the property that every spanning subgraph of $G$ with minimum
$d$-degree at least $\left(1+\varepsilon\right)m_{d}\left(k,n\right)p$
has a perfect matching. One interesting aspect of our proof is a ``non-constructive''
application of the absorbing method, which allows us to prove a bound
in terms of $m_{d}\left(k,n\right)$ without actually knowing its
value.
\end{abstract}

\section{Introduction}

Over the last few decades, there has been a great deal of interest in
analogues of combinatorial theorems \emph{relative to a random set}.
To give a simple example, let us consider \emph{Mantel's theorem}~\cite{Man07}, a classical theorem asserting that any subgraph
of the complete $n$-vertex graph $K_{n}$ with more than about half
of the $\binom{n}{2}$ possible edges must contain a triangle. The
random analogue of Mantel's theorem says that if one considers a \emph{random} subgraph $G\subseteq K_{n}$, obtained by including each edge
independently at random with some suitable probability $0<p<1$, then
typically $G$ has the property that each subgraph with more than
about half of the edges of $G$ must contain a triangle. That is to
say, Mantel's theorem is ``robust'' in the sense that an analogous
statement typically holds even in the ``noisy environment'' of a
random graph. The study of combinatorial theorems relative to random sets has been closely related to several of the most exciting recent developments in probabilistic and extremal combinatorics, including the sparse regularity method, hypergraph containers and the absorbing method. See \cite{Con14} for a general survey of this topic.

In the early history of this area, the available methods were somewhat ad-hoc, but recent years have seen the development of some very general tools and techniques that allow one to ``transfer'' a wide variety of combinatorial theorems to the random setting, without actually needing to know the details of their proofs. As an illustration of this, consider the hypergraph\footnote{A \emph{$k$-uniform hypergraph}, or \emph{$k$-graph} for short, is a pair $H=(V,E)$, where $V$
is a finite set of \emph{vertices}, and $E$
is a family of $k$-element subsets of $V$, referred to as the \emph{edges} of
$H$. Note that a $2$-graph is just a graph.} \emph{Tur\'an problem}, which is a vast generalisation of Mantel's problem. For a $k$-graph $H$, let $\operatorname{ex}(n,H)$ be the maximum possible number of edges in an $n$-vertex $k$-graph which contains no copy of $H$, and define the \emph{Tur\'an density} of $H$ as
$$\pi(H)=\lim_{n\to \infty} \frac{\operatorname{ex}(n,H)}{\binom n k}$$
(a simple monotonicity argument shows that this limit exists). In the graph case, where $k=2$, the values of each $\pi(H)$ are given by the celebrated Erd\H os--Stone--Simonovits theorem~\cite{ES66,ES46}, but for higher uniformities very little is known about the values of $\pi(H)$. Despite this, Conlon and Gowers~\cite{CG16} and independently Schacht~\cite{Sch16} were able to prove an optimal theorem in the random setting. Let $G\sim {\operatorname H}^{k}\left(n,p\right)$ be an instance of the random $k$-graph with edge probability $p$; they proved that if $p$ is not too small then typically $G$ has the property that every subgraph with at least $(\pi_k(H)+\varepsilon)\binom{n}{k}p$ edges has a copy of $H$. Moreover, they were able to find the optimal range of $p$ (that is, the essentially best possible definition of ``not too small'') for which this holds.

In other words, whenever we are able to prove a Tur\'an-type theorem for graphs or hypergraphs, we ``automatically'' get a corresponding theorem in the random setting. This was quite a striking development: before this work, similar theorems were known only for a few graphs $H$ (and no higher-uniformity hypergraphs), despite rather a lot of effort. For a more detailed history of this problem we refer the reader to \cite{CG16,Sch16} and the references therein.

The methods and tools developed by Conlon and Gowers, and by Schacht, were later supplemented by some further work by Conlon, Gowers, Samotij and Schacht~\cite{CGSS14}. The ideas developed by these authors are very powerful (and actually apply in much more general settings than just Tur\'an-type problems), but a common shortcoming is that none of them are sensitive to ``local'' information about the individual vertices of a graph or hypergraph, and therefore they are not sufficient for proving relative versions of theorems in which one wishes to understand the presence of \emph{spanning} substructures.

For example, \emph{Dirac's theorem}~\cite{Dir52} famously asserts that every $n$-vertex graph with minimum degree at least $n/2$ has a \emph{Hamiltonian cycle}: a cycle passing through all the vertices of the graph. A random analogue of this theorem was conjectured by Sudakov and Vu~\cite{SV08} and proved by Lee and Sudakov~\cite{LS12} (see also the refinements in \cite{Mon17,NST19}): For any $\varepsilon>0$, if $p$ is somewhat greater than $\log n/n$, then a random graph $\operatorname{H}^2(n,p)$ typically has the property that every spanning subgraph with minimum degree at least $(1+\varepsilon)np/2$ has a Hamiltonian cycle. Since this work, there has been a lot of interest in random versions of Dirac-type theorems for other types of spanning or almost-spanning subgraphs (see for example \cite{ABET20,ABHKP16,BCS11,BKS11,BKT13,FNNP17,FK08,HLS12,Mon20,NS17,SST18}), introducing a large number of ideas and techniques that are quite independent of the aforementioned general tools. In this paper we are interested in Dirac-type problems for random \emph{hypergraphs}. Before discussing this further, we take a moment to make some definitions and introduce the topic of (non-random) Dirac-type problems for hypergraphs.

Recall that Dirac's theorem asserts that every $n$-vertex graph with minimum degree at least $n/2$ has a Hamiltonian cycle. If $n$ is even then we can take every second edge on this cycle to obtain a \emph{perfect matching}: a set of vertex-disjoint edges that covers all the vertices of our graph. So, Dirac's theorem can also be viewed as a theorem about the minimum degree required to guarantee a perfect matching. While there are certain generalisations of (Hamiltonian) cycles to hypergraphs\footnote{One of these generalisations is called a \emph{Berge} cycle. Actually Clemens, Ehrenm\"uller and Person~\cite{CEP19} recently proved a generalisation of Dirac's theorem, and a random version of this theorem, for Hamiltonian Berge cycles.}, the notion of a perfect matching generalises unambiguously, and we prefer to focus on perfect matchings when considering hypergraphs of higher uniformities.

One subtlety is that in the hypergraph setting there are actually multiple possible generalisations of the notion of minimum degree. For a $k$-graph $H=(V,E)$ and a subset $S\subseteq V$ of the vertices of $H$, satisfying $0\leq|S|\leq k-1$, we define the \emph{degree} $\deg_{H}(S)$ of $S$ to be the number of edges of $H$ which include $S$. The \emph{minimum $d$-degree} $\delta_d(H)$ of $H$ is then defined to be the minimum, over all $d$-sets of vertices $S$, of $\deg_{H}(S)$. For integers $n,k,d$ such that $1\leq d\leq k-1$ and $n$ is divisible by $k$,
let $m_{d}(k,n)$ be the smallest integer $m$ such that every
$n$-vertex $k$-graph $H$ with $\delta_{d}(H)\geq m$ has a perfect matching. Dirac's theorem says that $m_1(2,n)\le \ceil{n/2}$, and it is quite easy to see that this is tight.

The problem of determining or approximating the values of $m_{d}(k,n)$ is fundamental in extremal graph theory, and has attracted a lot of attention in the last few decades (see for example the surveys~\cite{RR10, Zha16} and the references therein). The main conjecture in this area is as follows.

\begin{conjecture}
\label{conj:dirac-matching}
For fixed positive integers $d< k$, we have
\[
m_{d}(k,n)=\left(\max\left\{ \frac{1}{2},1-\left(1-\frac{1}{k}\right)^{k-d}\right\} +o(1)\right)\binom{n-d}{k-d},
\]
where $o(1)$ represents some error term that tends to zero as $n$
tends to infinity along some sequence of integers divisible by $k$.
\end{conjecture}

There are constructions showing that the expression in \cref{conj:dirac-matching} is a lower bound for $m_{d}(k,n)$; the hard part is to prove upper bounds.

In much the same way that the Tur\'an densities $\pi(H)$ encode the asymptotic behaviour of the extremal numbers $\operatorname{ex}(H,n)$, it makes sense to define \emph{Dirac thresholds} that encode the asymptotic behaviour of the values of $m_{d}(k,n)$. However, compared to the Tur\'an case, convergence to a limit is nontrivial; in \cref{sec:convergence} we prove the following result (which will be helpful to state and prove our main result, but is also of independent interest).

\begin{thm}
\label{thm:convergence}Fix positive integers $d<k$. Then the quantity
$m_{d}\left(k,n\right)/\binom{n-d}{k-d}$ converges to a limit $\mu_{d}\left(k\right)\in\left[0,1\right]$,
as $n$ tends to infinity along the positive integers divisible by
$k$.
\end{thm}

Note that \cref{conj:dirac-matching} can then be viewed as a conjecture for the values of the Dirac thresholds $\mu_{d}\left(k\right)$. This conjecture seems to be very difficult, but it has been proved in some special cases: namely, when $5d \ge 2k-2$, and when $(d,k)\in \{(1,4),(1,5)\}$ (see \cite{AFHRRS12,HPS09, Kha16,LM15,Pik08,RRS09,FK18}). A number of different upper and lower bounds have also been proved in various cases.

So, the situation is quite similar to the hypergraph Tur\'an problem: the optimal theorems in the non-random setting are not known, but there is still some hope of proving a ``transference'' theorem, giving bounds in the random setting in terms of the (unknown) Dirac thresholds $\mu_{d}\left(k\right)$.

Of course, in order to prove a random analogue of any extremal theorem, in addition to having a handle on the extremal theorem one also needs to have a good understanding of random graphs and hypergraphs. This presents a rather significant obstacle when investigating perfect matchings, because the study of perfect matchings in random hypergraphs is notoriously difficult. Famously, \emph{Shamir's problem} asks for which $p$ a random $k$-graph ${\operatorname H}^{k}(n,p)$ has a perfect matching, and this was resolved only a few years ago in a tour-de-force by Johansson, Kahn and Vu~\cite{JKV08} (see also the new simpler proof in \cite{FKNP19}, and the refinement in \cite{Kah19}). Roughly speaking, they proved that if $p$ is large enough that ${\operatorname H}^{k}(n,p)$ typically has no isolated vertices (the threshold value of $p$ is about $n^{1-k}\log n$), then ${\operatorname H}^{k}(n,p)$ typically has a perfect matching. All known proofs of this theorem are quite ``non-constructive'', involving some ingenious way to show that a perfect matching is likely to exist without being able to say much about its properties or how to find it.

In any case, it is natural to make the following conjecture, ``transferring'' Dirac-type theorems to random hypergraphs.

\begin{conjecture} \label{conj:resilience}
Fix $\gamma>0$ and positive integers $d<k$, and consider any $0<p<1$ (which may be a function of $n$). Suppose that $n$ is divisible by $k$. Then a.a.s.\footnote{By ``asymptotically almost surely'', or ``a.a.s.'', we mean that
the probability of an event is $1-o\left(1\right)$. Here and for
the rest of the paper, asymptotics are as $n\to\infty$, unless stated otherwise.}\ $G\sim{\operatorname H}^{k}\left(n,p\right)$ has the property that every spanning subgraph $G'\subseteq G$ with $\delta_d(G')\geq (\mu_d(k)+\gamma)\binom{n-d}{k-d}p$
has a perfect matching.
\end{conjecture}

Note that in the above conjecture we do not make any assumption on $p$, though in some sense we are implicitly assuming $p=\Omega(n^{d-k}\log n)$, because otherwise one can show that a random $k$-graph $G\sim{\operatorname H}^{k}\left(n,p\right)$ will a.a.s.\ have $\delta_d(G)=0$ (meaning that there is no subgraph $G'$ satisfying the condition in the conjecture). Due to the aforementioned difficulty of studying perfect matchings in random hypergraphs, we believe that \cref{conj:resilience} will be extremely difficult to prove for small $p$ (especially for $p\approx n^{1-k}\log n$), and therefore we believe that the hardest (and most interesting) case is where $d=1$. On the other extreme, if $d=k-1$ then it suffices to consider the regime where $p=\Omega(n^{-1}\log n)$, which is substantially easier due to certain techniques which allow one to reduce the problem of finding hypergraph perfect matchings to the problem of finding perfect matchings in certain bipartite graphs\footnote{While hypergraph matchings are in general not well understood, there are a number of extremely powerful tools available for studying matchings in bipartite graphs (such as Hall's theorem).}. Using such a reduction, the $d=k-1$ case of  \cref{conj:resilience} was proved by Ferber and Hirschfeld~\cite{FH19}.

Our main result in this paper is the following substantial progress towards \cref{conj:resilience}, proving it for all $d<k$ under certain restrictions on $p$ (even though the values of $\mu_d(k)$ are in general unknown).
 
\begin{thm}
\label{thm:resilience}Fix $\gamma>0$ and positive integers $d< k$. Then there
is some $C>0$ such that the following holds. Suppose that $p\ge \max\{n^{-k/2+\gamma},C n^{-k+2}\}$,
and that $n$ is divisible by $k$. Then a.a.s.\ $G\sim{\operatorname H}^{k}\left(n,p\right)$ has the property that every spanning subgraph $G'\subseteq G$ with $\delta_d(G')\geq (\mu_d(k)+\gamma)\binom{n-d}{k-d}p$
has a perfect matching.
\end{thm}

Recalling the implicit assumption $p=\Omega(n^{d-k}\log n)$, \cref{thm:resilience} actually resolves the $d>k/2$ case of \cref{conj:resilience}, and comes very close to resolving the case $d=k/2$ (if $d$ is even). Also, note that except in the case where $k=3$ and $d=1$, the assumption $p\ge C n^{-k+2} \log n$ is superfluous (being satisfied automatically when $p=\Omega(n^{d-k}\log n)$ and $p\ge n^{-k/2+\gamma}$). Actually, this particular assumption can be weakened quite substantially, but in the interest of presenting a clear proof, we discuss how to do this only informally, in \cref{sec:concluding}.

There are a number of different ideas and ingredients that go into the proof of \cref{thm:resilience}. Perhaps the most crucial one is a \emph{non-constructive} way to apply the so-called \emph{absorbing method}. To say just a few words about the absorbing method: in various different contexts, it is much easier to find \emph{almost}-spanning substructures than genuine spanning substructures. For example, a perfect matching is a collection of disjoint edges that cover all the vertices of a hypergraph, but it is generally much easier to find a collection of disjoint edges that cover \emph{almost} all of the vertices of a hypergraph. The insight of the absorbing method is that one can sometimes find small ``flexible'' substructures called \emph{absorbers}, arranged in a way that allows one to make local modifications to transform an almost-spanning structure into a spanning one. This method was pioneered by Erd\H os, Gy\'arf\'as and Pyber~\cite{EGP91}, and was later systematised by R\"odl, Ruci\'nski and Szemer\'edi~\cite{RRS06,RRS09}, in connection with their study of Dirac-type theorems in hypergraphs.

In previous work, the typical approach was to build absorbers in a ``bare-hands'' fashion, considering some set of vertices which we would like to be able to ``absorb'', and reasoning about the possible incidences between edges close to these vertices in order to prove that an appropriate absorber is present. For this to be possible, one must define the notion of an absorber in a very careful way. In contrast, further developing some ideas that we introduced in \cite{FK19}, we are able to find absorbers using a ``contraction'' argument, together with one of the general tools developed by Conlon, Gowers, Samotij and Schacht~\cite{CGSS14}. This gives us an enormous amount of freedom, and in particular we can define absorbers in terms of the Dirac thereshold $\mu_{d}\left(k\right)$ (without knowing its value!). This freedom is also crucial in allowing us to choose absorbers which exist in ${\operatorname H}^{k}\left(n,p\right)$ for small $p$ (that is, for $p$ close to $n^{-k/2}$, which seems to be the limit of our approach).

The structure of the rest of the paper is as follows. First, in \cref{sec:outline} we give an introduction to the absorbing method, and outline the proof of \cref{thm:resilience}. Afterwards, we present a short proof of \cref{thm:convergence} in \cref{sec:convergence}, as a warm-up to the absorbing method before we present the more sophisticated ideas in the proof of \cref{thm:resilience}.

In \cref{sec:regularity} we discuss the so-called sparse regularity method, and in \cref{sec:concentration} we record some basic facts about concentration of the edge distribution in random hypergraphs. Everything in these sections will be quite familiar to experts. In \cref{sec:almost-perfect} we explain how to find almost-perfect matchings in the setting of \cref{thm:resilience}, in \cref{sec:sparse-absorption} we state a sparse absorbing lemma and explain how to use it to prove \cref{thm:resilience}, and in \cref{sec:absorbers} we present the proof of this sparse absorbing lemma.

Finally, in \cref{sec:concluding} we have some concluding remarks, including a discussion of how to weaken the assumption $p\ge C n^{-k+2}$ in the case $(d,k)=(1,3)$.

\begin{rem}[added in proof]
The general approach of defining absorbers in terms of a Dirac threshold has also appeared in earlier work by Glock, K\"{u}hn, Lo, Montgomery and Osthus~\cite{GKLMO19}. We thank Stefan Glock for bringing this to our attention.
\end{rem}

\section{Outline of the proof of the main theorem}\label{sec:outline}

Suppose that $G\sim{\operatorname H}^{k}\left(n,p\right)$ is a typical outcome of ${\operatorname H}^{k}\left(n,p\right)$, and $G'\subseteq G$
is a spanning subgraph of $G$ with minimum $d$-degree at least $\left(\mu_{d}\left(k\right)+\gamma\right)p\binom{n-d}{k-d}$.
Our goal is to show that $G'$ contains a perfect matching. Since the proof is quite involved, we break down the steps of the proof into subsections.

\subsection{Almost-perfect matchings}

The first observation is that our task is much simpler if we relax
our goal to finding an \emph{almost}-perfect matching (that is, a matching
that covers all but $o\left(n\right)$ vertices). This is due to the
existence of a powerful tool called the \emph{sparse regularity lemma}.
Roughly speaking, the sparse regularity lemma allows us to model the
large-scale structure of the sparse $k$-graph $G'$ using a small,
dense $k$-graph $\mathcal{R}$ called a \emph{cluster $k$-graph}.
Each edge of $\mathcal{R}$ corresponds to a $k$-partite subgraph of
$G'$ where the edges are distributed in a ``homogeneous'' or ``quasirandom''
way\footnote{Hypergraph regularity lemmas of the type we use here are sometimes
known as \emph{weak} regularity lemmas, to distinguish them from a
much stronger and more complicated hypergraph regularity lemma which
does not permit a description in terms of cluster $k$-graphs.}.

It is not hard to show that the degree condition on $G'$ translates to a similar degree condition
on $\mathcal{R}$, though small errors are introduced in the process:
we can show that almost all of the $d$-sets of vertices in $\mathcal{R}$
have degree at least say $\left(\mu_{d}\left(k\right)+\gamma/2\right)\binom{t-d}{k-d}$,
where $t$ is the number of vertices of $\mathcal{R}$. We then use
the definition of $\mu_{d}\left(k\right)$ (without knowing its value!)
to show that $\mathcal{R}$ has an almost-perfect matching. This is not immediate,
because $\mathcal{R}$ may have a few $d$-sets of vertices with small
degree, but it is possible to use a random sampling argument to overcome
this difficulty. In any case, an almost-perfect matching in $\mathcal{R}$
tells us how to partition most of the vertices of $G'$ into subsets
such that the subgraphs induced by these subsets each satisfy a certain
quasirandomness condition. We can then take advantage of this quasirandomness
to find an almost-perfect matching in each of the subgraphs. Combining
these matchings gives an almost-perfect matching in $G'$.

The details of this argument are in \cref{sec:almost-perfect}.

\subsection{The absorbing method\label{subsec:absorbing-method}}

It may not be obvious that being able to find almost-perfect matchings
is actually useful, if our goal is to find a perfect matching. It
is certainly not true that we can start from any almost-perfect matching
and add a few edges to obtain a perfect matching. However, it turns
out that something quite similar is often possible in problems of
this type. Namely, in some hypergraph matching problems it is possible
to find a small subset of vertices $X$ which is very ``flexible''
in the sense that it can contribute to matchings in many different
ways. We can then find an almost-perfect matching covering almost
all the vertices outside $X$, and take advantage of the special properties
of $X$ to complete this into a perfect matching. This idea is now
called the \emph{absorbing method}. It was introduced as a general
method by R\"odl, Ruci\'nski and Szemer\'edi~\cite{RRS06,RRS09} (though similar
ideas had appeared earlier, for example by Erd\H os, Gy\'arf\'as and Pyber~\cite{EGP91} and by Krivelevich~\cite{Kri97}). The absorbing
method has been an indispensable tool for almost all work on hypergraph
matching problems in the last decade.

To give a specific example, the \emph{strong absorbing lemma} of H\'an,
Person and Schacht~\cite{HPS09} (appearing here as \cref{lem:strong-absorbing-lemma}) shows that in a very dense $k$-graph $G$
we can find a small ``absorbing'' set of vertices $X$, with the
special property that for any set $W$ of $o\left(n\right)$ vertices
outside $X$, the induced subgraph $G\left[X\cup W\right]$ has a
perfect matching. So, if we can find an almost-perfect matching $M_{1}$
in $G-X$, we can take $W$ as the set of unmatched vertices and use
the special property of $X$ to find a perfect matching $M_{2}$ in
$G\left[X\cup W\right]$, giving us a perfect matching $M_{1}\cup M_{2}$
in $G$.

It is much more difficult to prove absorbing lemmas in the sparse
setting of \cref{thm:resilience}. To explain why, we need to say a bit more about how
absorbing lemmas are proved in the dense setting. Almost always, the
idea is to build an absorbing set $X$ using small subgraphs called
\emph{absorbers}\footnote{The language in this field has still not been fully standardised.
For example, in \cite{HPS09} the authors use the term ``absorbing $m$-set''
instead of ``absorber''.}. In the context of matching problems in $k$-graphs, an absorber
in a $k$-graph $G$ rooted at a $k$-tuple of vertices $x_{1},\dots,x_{k}$
is a subgraph $H$ whose edges can be partitioned into two matchings,
one of which covers every vertex in $V\left(H\right)$ and the other
of which covers every vertex except $x_{1},\dots,x_{k}$. A single
edge $\left\{ x_{1},\dots,x_{k}\right\} $ is a trivial absorber,
and in the case $k=2$ (that is, the case of graphs), an odd-length
path between $x_{1}$ and $x_{2}$ is an absorber. See \cref{fig:3-absorber} for a nontrivial
example of a 3-uniform absorber.

\begin{figure}[h]
\begin{center}
\includegraphics{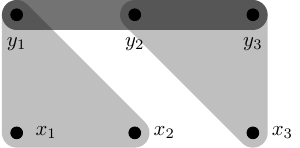}
\end{center}

\caption{\label{fig:3-absorber}An illustration of a 3-uniform absorber rooted on vertices $x_{1},x_{2},x_{3}$.
The dark edge covers all non-root vertices and the two light edges
form a matching covering all the vertices of the absorber.}

\end{figure}
The details in the proofs of different absorbing lemmas vary somewhat,
but a common first step is to show that there are many absorbers rooted
at every $k$-tuple of vertices, using fairly ``bare-hands'' arguments
that take advantage of degree assumptions. For example, suppose an
$n$-vertex 3-graph $G$ has $\delta_{2}\left(G\right)\ge\left(1/2+\gamma\right)n$,
consider any vertices $x_{1},x_{2},x_{3}$, and suppose we are trying
to find a copy of the absorber pictured in \cref{fig:3-absorber}. There are at least $\left(1/2+\gamma\right)n$
choices for $y_{1}$ such that $\left\{ x_{1},x_{2},y_{1}\right\} \in E\left(G\right)$.
For any such $y_{1}$, and any of the $n-4$ remaining choices of
$y_{2}$, there are at least $\left(1/2+\gamma\right)n$ choices for
$y_{3}$ such that $\left\{ y_{1},y_{2},y_{3}\right\} \in E\left(G\right)$,
and at least $\left(1/2+\gamma\right)n$ choices such that $\left\{ x_{3},y_{2},y_{3}\right\} \in E\left(G\right)$,
so by the inclusion-exclusion principle there are at least $2\gamma n$
choices for $y_{3}$ such that both $\left\{ y_{1},y_{2},y_{3}\right\} $
and $\left\{ x_{3},y_{2},y_{3}\right\} $ are in $E\left(G\right)$.
All in all, this gives about $\gamma n^{3}\approx3\gamma\binom{n}{3}$
absorbers rooted at $x_1,x_2,x_3$.

Having shown that every $k$-tuple of vertices supports many absorbers,
one can then often use a straightforward probabilistic argument to
construct an arrangement of absorbers that gives rise to an absorbing
set $X$ as in the strong absorbing lemma. Continuing with the previous
example, if we choose a random set $T$ of say $\left(\gamma/2\right)n$ disjoint triples of vertices,
then for every choice of $x_{1},x_{2},x_{3}$, there are typically
about $\left(3\gamma\right)\left(\gamma/2\right)n$ triples in $T$
which give an absorber rooted on $x_{1},x_{2},x_{3}$. We can then
take $X$ as the set of vertices in the triples in $T$. It is not hard to
check that this set satisfies the assumptions of the strong absorbing
lemma: given a set $W$ of $o\left(n\right)$ vertices outside $X$,
we can partition $W$ into triples $\left\{ x_{1},x_{2},x_{3}\right\} $
and iteratively ``absorb'' them into $T$ to obtain a perfect matching
in $G\left[X\cup W\right]$.

\subsection{Finding absorbers in sparse graphs}

Unfortunately, the ideas sketched above fail in many different ways
in the sparse setting. First, there is the problem of how to actually
find absorbers. It is in general very difficult to understand when
one can find a copy of a specific $k$-graph in a subgraph $G'$ of
a random $k$-graph $G$. Indeed, this is the random Tur\'an problem
described in the introduction, and general results have become available
only very recently. One of the most flexible tools in this area is
the \emph{sparse embedding lemma} proved by Conlon, Gowers, Samotij
and Schacht~\cite{CGSS14} (previously, and sometimes still, known as the \emph{K\L R
conjecture} of Kohayakawa, \L uczak and R\"odl). 

Roughly speaking, the sparse embedding lemma says that for any\footnote{We are not being completely truthful here: strictly speaking, $H$ must be a so-called \emph{linear} $k$-graph, but this restriction turns out not to be particularly important for us.}
$k$-graph $H$, if $p$ is large enough that a random $k$-graph
$G\sim{\operatorname H}^{k}\left(n,p\right)$ typically contains many copies of $H$ (this depends on
a ``local sparseness'' measure of $H$ called \emph{$k$-density}), then $G$ satisfies the
following property: If we apply the sparse regularity lemma to a spanning
subgraph $G'\subseteq G$, and find a copy of $H$ in the resulting
cluster graph $\mathcal{R}$, then there is a corresponding copy of
$H$ in $G'$ itself. Roughly speaking, the sparse embedding lemma allows us to work
in the dense cluster graph, where it is much easier to reason directly
about existence of subgraphs, and then ``pull back'' our findings
to the original graph.

One may hope that we can just repeat the arguments in the proof of
the strong absorbing lemma to find absorbers in the dense cluster
graph, and somehow use the sparse embedding lemma to convert these
into absorbers in the original graph. Unfortunately, life is not this
simple, for (at least...) two reasons. The first issue is that we need our absorbers to satisfy some local sparseness condition, because otherwise we can only work with a very limited range of $p$. It is not obvious how to use existing ``bare-hands'' methods to find such absorbers.

The second issue is that the sparse embedding lemma is not suited
for embedding \emph{rooted} subgraphs. The cluster graph $\mathcal{R}$ is just too rough
a description of $G'$ for it to be possible to deduce information
about specific vertices in $G'$ from information in $\mathcal{R}$.

To attack the first of these issues, we use a novel \emph{non-constructive}
method to find our absorbers. Namely, since an absorber
is built out of matchings, we can use the definition of the Dirac
threshold itself to find absorbers (even if we do not actually know
its value). To be more specific, consider a $k$-graph $G$ with $\delta_{d}\left(G\right)\ge\left(\mu_{d}\left(k\right)+\gamma\right)\binom{n-d}{k-d}$,
and let $M$ be a large constant. Using a concentration inequality,
we can show that that almost all $M$-vertex induced subgraphs
of $G$ have minimum $d$-degree at least $\left(\mu_{d}\left(k\right)+\gamma/2\right)\binom{M-d}{k-d}$,
so if $M$ is large enough, then almost all $M$-vertex induced subgraphs
have a perfect matching. We then have a lot of freedom to construct
a locally sparse absorber using these matchings (specifically, we construct an absorber using an explicit locally sparse ``pattern'' graph).

To overcome the second of the aforementioned issues, we further develop
a ``contraction'' technique we introduced in~\cite{FK19}. The problem is
that the cluster graph does not ``see'' individual vertices; it can
only see large sets of vertices. In the case $k=2$ (that is, the
graph case), an obvious fix would be to consider the set of \emph{neighbours
}(or perhaps neighbours-of-neighbours) of our desired roots, instead
of the roots themselves. However, in the case $k\ge3$, every edge
containing a root vertex $x_{i}$ contains $k-1>1$ other vertices.
That is to say, ``neighbours'' come grouped in sets of size $k-1$
(the collection of all such sets is called the \emph{link $\left(k-1\right)$-graph
}of $x_{i}$). So, it seems we would need an embedding lemma that
works with sets of $(k-1)$-sets of vertices, not just sets of vertices.

The way around this problem is to choose a large matching in the link
$\left(k-1\right)$-graph of each $x_{i}$, and ``contract'' each
of the edges in each of these matchings to a single vertex, to obtain
a contracted graph $G_{\textrm{cont}}'$. If we do this carefully,
the resulting graph can still be viewed as a subgraph of an appropriate
random $k$-graph, so the sparse embedding lemma still applies. It
then suffices to find a suitable ``contracted absorber'' in $G_{\textrm{cont}}'$,
which would correspond to an absorber in the original $k$-graph $G'$.
We can do this with the sparse embedding lemma.

The details of the arguments sketched in this section appear in \cref{sec:absorbers}.

\subsection{Combining the absorbers}

The above discussion gives a rough idea for how to find an absorber rooted at every
$k$-set of vertices, in a suitable spanning subgraph of a random
graph. However, it is still not at all obvious how to combine these
to prove a sparse absorbing lemma. A simple probabilistic argument
as sketched in \cref{subsec:absorbing-method} cannot suffice: unfortunately, there are just not enough absorbers.

We get around the issue as follows. Instead of using our absorbers
to find a matching $M$ which can ``absorb'' every $k$-set of vertices,
we fix a specific ``template'' arrangement of only linearly many $k$-sets we would
like to be able to absorb. It is easy to handle such a small number
of $k$-sets: we can in fact greedily choose \emph{disjoint} absorbers
for each of these special $k$-sets, to obtain an ``absorbing structure''
$H$. Building on ideas due to Montgomery~\cite{Mon14,Mon19}, we show that it is
possible to choose our template arrangement of $k$-sets in such a way that
$H$ has a very special kind of robust matching property: $H$ has
a ``flexible set'' of vertices $Z$ such that $H$ still has a perfect
matching even after any constant fraction of the
vertices in $Z$ are deleted\footnote{Various authors have coined different names for different ways to apply the absorbing method (though these names and their usage do not always seem to be completely consistent). The use of a flexible set $Z$ which can optionally contribute to a desired structure is often called the \emph{reservoir method}, where $Z$ is called a \emph{reservoir}. In particular, Montgomery's approach, in which an absorbing structure is built using a template with a robust matching property, is often called \emph{distributive absorption}, or sometimes the \emph{absorber-template method}.}.

We can then let $X=V\left(H\right)$, and prove that $X$ gives a
sparse absorbing lemma, as follows. For any small set $W$ of vertices outside $X$, we can first find a matching $M_{1}$ covering
$W$ and a constant fraction of $Z$, using a hypergraph matching criterion
due to Aharoni and Haxell~\cite{AH00}. Then, our robust matching property implies
that $H-V(M_{1})$ has a perfect matching $M_{2}$, so $M_{1}\cup M_{2}$
is a perfect matching of $G\left[W\cup X\right]$.

The details of this argument, along with the statement of our sparse absorbing lemma and the deduction of \cref{thm:resilience}, are in \cref{sec:sparse-absorption}.

\section{Convergence of the Dirac threshold}\label{sec:convergence}

In this section we prove \cref{thm:convergence}, which will be a good
warm-up for some of the ideas that we will develop further to prove
\cref{thm:resilience}.

The main ingredient in the proof of \cref{thm:convergence} is the \emph{strong
absorbing lemma} due to H\'an, Person and Schacht~\cite[Lemma~2.4]{HPS09} (building on ideas of R\"odl, Ruci\'nski and Szemer\'edi~\cite{RRS06,RRS09}).
\begin{lem}
\label{lem:strong-absorbing-lemma}For any positive integers $d<k$,
and any $\gamma>0$, there is $n_{0}\in{\mathbb N}$ such that for every $n>n_{0}$
the following holds. Suppose that $G$ is a $k$-graph on $n$ vertices
with $\delta_{d}\left(G\right)\ge\left(1/2+\gamma\right)\binom{n-d}{k-d}$.
Then there is a set $X\subseteq V(G)$ such that
\begin{enumerate}
\item [(i)]$|X|\le \left(\gamma/2\right)^{k}n$, and
\item [(ii)]for every set $W\subseteq V\left(G\right)\backslash X$
of size at most $\left(\gamma/2\right)^{2k}n$ and divisible by $k$,
there is a matching in $G$ covering exactly the vertices of $X\cup W$.
\end{enumerate}
\end{lem}

Now, define
\[
\tilde{\mu}_{d}\left(k\right)=\liminf_{n\to\infty}\frac{m_{d}\left(k,n\right)}{\binom{n-d}{k-d}},
\]
where $n\to\infty$ along the integers $n$ divisible by $k$. Our
main goal is to prove the following lemma.
\begin{lem}
\label{lem:convergence-lemma}Fix positive integers $d<k$ and consider
any $\gamma>0$. Then for sufficiently large $n$ divisible by $k$, every $n$-vertex
$k$-graph $G$ with $\delta_{d}\left(G\right)\ge\left(\tilde{\mu}_{d}\left(k\right)+\gamma\right)\binom{n-d}{k-d}$
has a perfect matching.
\end{lem}

Before we explain how to prove \cref{lem:convergence-lemma} we show
how it implies \cref{thm:convergence}.
\begin{proof}[Proof of \cref{thm:convergence}, given \cref{lem:convergence-lemma}]
\cref{lem:convergence-lemma} implies that $m_{d}\left(k,n\right)/\binom{n-d}{k-d}\le\tilde{\mu}_{d}\left(k\right)+\gamma$
for sufficiently large $n$ (divisble by $k$), and since $\gamma>0$ was arbitrary it
follows that 
\[
\limsup_{n\to\infty}\frac{m_{d}\left(k,n\right)}{\binom{n-d}{k-d}}\le\tilde{\mu}_{d}\left(k\right)=\liminf_{n\to\infty}\frac{m_{d}\left(k,n\right)}{\binom{n-d}{k-d}},
\]
from which it follows that $m_{d}\left(k,n\right)/\binom{n-d}{k-d}$
converges to a limit $\mu_{d}\left(k\right)=\tilde{\mu}_{d}\left(k\right)$.
\end{proof}
Now, our proof of \cref{lem:convergence-lemma} will consist of two
steps. First, we prove that the conditions
of \cref{lem:convergence-lemma} ensure an \emph{almost-perfect} matching,
then we will use the strong absorbing lemma (\cref{lem:strong-absorbing-lemma})
to transform this into a perfect matching. The following lemma encapsulates
the first of these steps.
\begin{lem}
\label{lem:almost-perfect-convergence}Fix positive integers $d<k$
and consider any $\eta>0$. Then for sufficiently large $n$,
every $n$-vertex $k$-graph $G$ with $\delta_{d}\left(G\right)\ge\left(\tilde{\mu}_{d}\left(k\right)+\eta\right)\binom{n-d}{k-d}$
has a matching covering all but $o\left(n\right)$ vertices.
\end{lem}

To prove \cref{lem:almost-perfect-convergence} we need the following
lemma showing that random subgraphs of hypergraphs typically inherit minimum-degree conditions. We state this in a slightly more general form than we need here, for later use.
\begin{lem}
\label{lem:random-subset-degrees}There is $c=c(k)>0$ such that the following holds. Consider an $n$-vertex $k$-graph
$G$ where all but $\delta\binom{n}{d}$ of the $d$-sets
have degree at least $\left(\mu+\eta\right)\binom{n-d}{k-d}$.
Let $S$ be a uniformly random subset of $Q\ge 2d$ vertices of $G$. Then with probability
at least $1-\binom{Q}{d}\left(\delta+e^{-c\eta^{2}Q}\right)$,
the random induced subgraph $G\left[S\right]$ has minimum $d$-degree
at least $\left(\mu+\eta/2\right)\binom{Q-d}{k-d}$.
\end{lem}

\begin{proof}
Let $W^{\left(d\right)}$ be the collection of $d$-sets with degree
less than $\left(\mu +\eta\right)\binom{n-d}{k-d}$ in $G$, and randomly order the vertices of $G$ as $v_1,\dots,v_n$, so we may take $S=\{v_{1},\dots,v_{Q}\}$. We will prove that
\[
\Pr\left(\deg_{S}\left(\left\{ v_{1},\dots,v_{d}\right\} \right)<\left(\mu+\eta/2\right)\binom{Q-d}{k-d}\right)\le\delta+e^{-c\eta^{2}Q}
\]
(where we abuse notation slightly and write $\deg_{S}\left(\left\{ v_{1},\dots,v_{d}\right\} \right)$ for the number of edges which contain $v_{1},\dots,v_{d}$ and $k-d$ vertices of $S$).
The desired result will then follow from symmetry and the union bound.

First note that the probability of the event $\left\{ v_{1},\dots,v_{d}\right\} \in W^{\left(d\right)}$
is at most $\delta$. Now, condition on any outcome of $\left\{ v_{1},\dots,v_{d}\right\}$ which is not in $W^{\left(d\right)}$. Then $\{v_{d+1},\dots,v_Q\}$ is a uniformly random subset of the vertices of $G$ other than $v_{1},\dots,v_{d}$, and $\E \deg_{S}\left(\left\{ v_{1},\dots,v_{d}\right\} \right)\ge (\mu+\eta)\binom{Q-d}{k-d}$. Also, making any ``swap'' to our subset $\{v_{d+1},\dots,v_Q\}$ (that is, exchanging any element with an element outside this subset) affects $\deg_{S}\left(\left\{ v_{1},\dots,v_{d}\right\} \right)$ by at most $\binom {Q-d-1}{k-d-1}$. So, by a concentration inequality such as \cite[Corollary~2.2]{GIKM17},
conditioned on our outcome of $\left\{ v_{1},\dots,v_{d}\right\} \notin W^{\left(d\right)}$,
the probability that $\left\{ v_{1},\dots,v_{d}\right\} $ has
degree less than $\left(\mu+\eta/2\right)\binom{Q-d}{k-d}$ in $G[S]$
is at most
$$2\exp\left(-\frac{2\left(\frac{\eta}2 \binom {Q-d}{k-d}\right)^2}{(Q-d)\binom {Q-d-1}{k-d-1}^2}\right)\le e^{-c\eta^{2}Q}$$
(for say $c=1/(4k^2)$, recalling that $Q\ge 2d$), as claimed.
\end{proof}

Now we prove \cref{lem:almost-perfect-convergence}.
\begin{proof}[Proof of \cref{lem:almost-perfect-convergence}]
Choose large $Q$, divisible by $k$, such that $m_{d}\left(k,Q\right)/\binom{Q-d}{k-d}\le\tilde{\mu}_{d}\left(k\right)+\eta/2$
(this is possible by the definition of $\tilde{\mu}_{d}\left(k\right)$). Let $\lambda=\binom{Q}{d}e^{-c\eta^{2}Q}$ be as in \cref{lem:random-subset-degrees} (taking $\delta=0$), and note that we can make $\lambda$ arbitrarily small by making $Q$ large.

Now, we randomly partition the vertex set into $n/Q$ subsets of
size $Q$. By \cref{lem:random-subset-degrees}, with positive probability
all but a $\lambda$-fraction of the subsets have minimum
degree at least $\left(\tilde{\mu}_{d}\left(k\right)+\eta/2\right)\binom{Q-d}{k-d}$.
By our choice of $Q$, each of these $Q$-vertex subsets $S\subseteq V(G)$ has the property that $G\left[S\right]$ has a
perfect matching, and we can combine these to find a matching covering
all but $\lambda n$ vertices. Since $\lambda$ could have
been arbitrarily small, this implies that we can find a matching covering
all but $o\left(n\right)$ vertices.
\end{proof}
Now, it is straightforward to deduce \cref{lem:convergence-lemma} from \cref{lem:strong-absorbing-lemma} and \cref{lem:almost-perfect-convergence},
concluding our proof of \cref{thm:convergence}.
\begin{proof}[Proof of \cref{lem:convergence-lemma}]
From the discussion in the introduction, note that $\tilde{\mu}_{d}\left(k\right)\ge1/2$,
so the assumptions in \cref{lem:strong-absorbing-lemma} are satisfied
and we can find an ``absorbing set'' $X\subseteq V(G)$ of at most
$\left(\gamma/2\right)^{k}n$ vertices. Let $n'=n-|X|$
and observe that since $X$ is so small, we have $\delta_{d}\left(G-X\right)\ge\left(\tilde{\mu}_{d}\left(k\right)+\gamma/2\right)\binom{n'-d}{k-d}$.
By \cref{lem:almost-perfect-convergence}, it follows that $G-X$ has
a matching covering all vertices except a set $W$ of size $o\left(n\right)$. By the defining property
of the absorbing set $X$, it follows that $G$ has a perfect
matching.
\end{proof}

\section{The sparse regularity method}\label{sec:regularity}

The proof of \cref{thm:resilience} makes heavy use of the sparse regularity method. So, we will need a sparse version of a hypergraph regularity lemma. There is a general hypergraph regularity lemma which is quite complicated to state and prove (see \cite{Gow07,RS04}), but we will only need (a sparse version of) the so-called ``weak'' hypergraph regularity lemma (see \cite{KNRS10}). Weak hypergraph regularity lemmas are suitable for embedding \emph{linear} hypergraphs, which are hypergraphs in which no pair of edges share more than one vertex.

We remark that this section closely mirrors \cite[Section~4]{FK19}, though some of the lemma statements are slightly more general.

To state our sparse hypergraph regularity lemma we first need to make some basic definitions.
\begin{defn}
Let $\varepsilon,\eta>0$, $D>1$ and $0\leq p\leq1$ be arbitrary parameters.
\begin{itemize}
\item \textbf{Density:}
Consider disjoint vertex sets $X_{1},\dots,X_{k}$ in a
$k$-graph $G$. Let $e\left(X_{1},\ldots,X_{k}\right)$ be the number of edges with a vertex in each $X_i$. Let
\[
d\left(X_{1},\ldots,X_{k}\right)=\frac{e\left(X_{1},\ldots,X_{k}\right)}{\left|X_{1}\right|\dots\left|X_{k}\right|}
\]
be the \emph{density} between $X_{1},\dots,X_{k}$.
\item \textbf{Regular tuples:} A $k$-partite $k$-graph with
parts $V_{1},\dots,V_{k}$ is \emph{$\left(\varepsilon,p\right)$-regular}
if, for every $X_{1}\subseteq V_{1},\dots,X_{k}\subseteq V_{k}$
with $\left|X_{i}\right|\ge\varepsilon\left|V_{i}\right|$, the density
$d\left(X_{1},\dots,X_{k}\right)$ of edges between $X_{1},\dots,X_{k}$
satisfies 
\[
\left|d\left(X_{1},\dots,X_{k}\right)-d\left(V_{1},\dots,V_{k}\right)\right|\le\varepsilon p.
\]
\item \textbf{Regular partitions:} A partition of the vertex set of a $k$-graph
into $t$ parts $V_{1},\dots,V_{t}$ is said to be $\left(\varepsilon,p\right)$-regular
if it is an equipartition (meaning that the sizes of the parts differ by at most one), and for all but at most $\varepsilon \binom {t}{k}$ of the
$k$-sets $\left\{V_{i_{1}},\dots,V_{i_{k}}\right\}$, the induced
$k$-partite $k$-graph between $V_{i_{1}},\dots,V_{i_{k}}$ is $\left(\varepsilon,p\right)$-regular.
\item \textbf{Upper-uniformity:} A $k$-graph $G$ is \emph{$\left(\eta,p,D\right)$-upper-uniform}
if for any choice of disjoint subsets $X_{1},\dots,X_{k}$ with $\left|X_{1}\right|,\dots,\left|X_{k}\right|\ge\eta\left|V\left(G\right)\right|$,
we have $d\left(X_{1},\dots,X_{k}\right)\le Dp$.
\end{itemize}
\end{defn}

Now, our sparse weak hypergraph regularity lemma is as follows. We omit
its proof since it is straightforward to adapt a proof of the sparse
graph regularity lemma (see \cite{KR03} for a sparse regularity lemma
for graphs, and see \cite[Theorem~9]{KNRS10} for a weak regularity
lemma for dense hypergraphs).
\begin{lem}
\label{lem:sparse-regularity}For every $\varepsilon,D>0$ and every
positive integer $t_{0}$, there exist $\eta>0$ and $T\in{\mathbb N}$ such
that for every $p\in\left[0,1\right]$, every $\left(\eta,p,D\right)$-upper-uniform
$k$-graph $G$ with at least $t_{0}$ vertices admits an $\left(\varepsilon,p\right)$-regular
partition $V_{1},\dots,V_{t}$ of its vertex set into $t_{0}\le t\le T$
parts.
\end{lem}

For us, the most crucial aspect of the sparse regularity lemma is that it can be used to give a rough description of a sparse $k$-graph in terms of a dense \emph{cluster $k$-graph} which we now define.
\begin{defn}
Given an $\left(\varepsilon,p\right)$-regular partition $V_{1},\dots,V_{t}$
of the vertex set of a $k$-graph $G$, the \emph{cluster hypergraph}
is the $k$-graph whose vertices are the clusters $V_{1},\dots,V_{t}$,
with an edge $\left\{ V_{i_{1}},\dots,V_{i_{k}}\right\} $ if $d\left(V_{i_{1}},\dots,V_{i_{k}}\right)>2\varepsilon p$
and the induced $k$-partite $k$-graph between $V_{i_{1}},\dots,V_{i_{k}}$ is $\left(\varepsilon,p\right)$-regular.
\end{defn}

If the sparse regularity lemma is applied with small $\varepsilon$ and large
$t_{0}$, the cluster hypergraph approximately inherits minimum degree
properties from the original graph $G$, as follows.
\begin{lem}
\label{lem:transfer-cluster-graph-sparse}Fix positive integers $d<k$, $0<\delta<1$, some sufficiently small
$\varepsilon>0$ and some sufficiently large $t_{0}\in{\mathbb N}$, and let $G$
be an $n$-vertex $\left(o(1),p,1+o\left(1\right)\right)$-upper-uniform
$k$-graph (in particular, we assume that $n$ is sufficiently large). Let $G'\subseteq G$ be a spanning subgraph in which all but $o(n^d)$ of the $d$-sets of vertices have degree at least $\delta \binom{n-d}{k-d}p$. Let $\mathcal{R}$ be the
$t$-vertex cluster $k$-graph obtained by applying the sparse regularity
lemma to $G'$ with parameters $t_{0}$, $p$ and $\varepsilon$.
Then all but at most $\sqrt{\varepsilon}\binom{t}{d}$ of the $d$-sets of vertices of $\mathcal{R}$ have degree at least $\delta \binom{t-d}{k-d}-(4\sqrt \varepsilon+k/t_0) t^{k-d}$.
\end{lem}

\cref{lem:transfer-cluster-graph-sparse} can be proved with a standard counting
argument. It is a special case of \cref{lem:transfer-cluster-graph-partition}, which we will state and prove in the next subsection.

\subsection{Refining an existing partition}

We will need to apply the sparse regularity lemma to a $k$-graph
whose vertices are already partitioned into a few different parts
with different roles. It will be important that the regular partition
in \cref{lem:sparse-regularity} can be chosen to be consistent with
this existing partition.
\begin{lem}
\label{lem:respect-partition}Suppose that a $k$-graph $G$ has its
vertices partitioned into sets $P_{1},\dots,P_{h}$. In the $\left(\varepsilon,p\right)$-regular
partition guaranteed by \cref{lem:sparse-regularity}, we can assume
that all but at most $\varepsilon ht$ of the clusters $V_{i}$ are
contained in some $P_{j}$.
\end{lem}

For the reader who is familiar with the proof of the regularity lemma,
the proof of \cref{lem:respect-partition} is straightforward. Indeed,
in order to prove the regularity lemma, one starts with an arbitrary
partition and iteratively refines it. Therefore, one can start with the partition $(P_1,\ldots,P_h)$ and proceed in the usual way. For more details, see the reduction
in \cite[Lemma~4.6]{FK19}.

Next, we state a more technical version of \cref{lem:transfer-cluster-graph-sparse}, deducing degree conditions in the cluster graph from degree conditions between the $P_{i}$. To state this we first need to generalise the definition of a cluster graph.
\begin{defn}
Consider a $k$-graph $G$, and let $P_{1},\dots,P_{h}$ be disjoint sets of vertices. Also, consider a $\left(\varepsilon,p\right)$-regular partition $V_{1},\dots,V_{t}$ of the vertices of $G$. Then the \emph{partitioned} cluster graph $\mathcal{R}$
with \emph{threshold} $\tau$ is the $k$-graph defined as follows.
The vertices of $\mathcal{R}$ are the clusters $V_{i}$ which are
completely contained in some $P_{j}$, with an edge $\left\{ V_{i_{1}},\dots,V_{i_{k}}\right\} $
if $d\left(V_{i_{1}},\dots,V_{i_{k}}\right)>\tau p$ and the induced
$k$-partite $k$-graph between $V_{i_{1}},\dots,V_{i_{k}}$ is $\left(\varepsilon,p\right)$-regular.
\end{defn}

\begin{lem}
\label{lem:transfer-cluster-graph-partition}Let $d<k$ be positive integers,
let $\varepsilon>0$ be sufficiently small, and let $t_{0}\in\mathbb{N}$
be sufficiently large. Let $G$ be an $n$-vertex $\left(o\left(1\right),p,1+o\left(1\right)\right)$-upper-uniform
$k$-graph with a partition $P_{1},\dots,P_{h}$ of its vertices into
parts of sizes $n_{1},\dots,n_{h}$, respectively. Let $G'\subseteq G$ be a spanning
subgraph and let $\mathcal{R}$ be the partitioned cluster $k$-graph with threshold
$\tau$ obtained by applying \cref{lem:sparse-regularity,lem:respect-partition}
to $G'$ with parameters $t_{0}$, $p$ and $\varepsilon$.

For every
$1\leq i\leq h$, let $\mathcal{P}_{i}$ be the set of clusters contained
in $P_{i}$, and let $t_{i}=|\mathcal P_i|$. Also, for every $J\subseteq \{1,\dots,h\}$, write $P_J:=\bigcup_{j\in J} P_j$, $n_J=|P_J|$, $\mathcal{P}_{J}=\bigcup_{j\in J} \mathcal P_j$ and $t_J=|\mathcal{P}_{J}|$. Then the following properties hold.

\begin{enumerate}
\item[(1)] Each $t_{i}\ge\left(n_{i}/n\right)t-\varepsilon ht$.
\item[(2)]{Consider some $i\le h$ and some $J\subseteq \{1,\dots,h\}$, and suppose that all but $o(n^d)$ of the $d$-sets of vertices $X\subset P_{i}$ satisfy
\[
\deg_{P_{J}}\left(X\right)\ge\delta p\binom{n_{J}-d}{k-d}.
\]
Then, all but at most $\sqrt{\varepsilon}\binom{t}{d}$
of the $d$-sets of clusters $\mathcal{X}\subset\mathcal{P}_{i}$
have 
\[
\deg_{\mathcal{P}_{J}}\left(\mathcal{X}\right)\ge\delta \binom{t_{J}-d}{k-d}-\left(\tau+\varepsilon h+\sqrt{\varepsilon}+k/t_{0}\right)t^{k-d}
\]
in the cluster graph $\mathcal{R}$.
}
\end{enumerate}
\end{lem}

Note that \cref{lem:transfer-cluster-graph-sparse} is actually a special
case of \cref{lem:transfer-cluster-graph-partition} (taking $h=1$ and threshold $\tau = 2\varepsilon$).

\begin{proof}
The clusters in $\mathcal{P}_{i}$ comprise at most $t_{i}\left(n/t\right)$
vertices, so recalling the statement of \cref{lem:sparse-regularity,lem:respect-partition},
we have $\left|P_{i}\right|=n_{i}\le t_{i}\left(n/t\right)+\varepsilon ht\left(n/t\right)$.
It follows that $t_{i}\ge\left(n_{i}/n\right)t-\varepsilon ht$, proving (1).

Now we prove (2). Let $\mathcal W^{(d)}$ be the collection of all $d$-sets of clusters $\mathcal W=\{W_1,\dots,W_d\}$ that are contained in more than $\sqrt{\varepsilon}\binom{t-d}{k-d}$ irregular (that is, non-$\left(\varepsilon,p\right)$-regular) $k$-sets $\left\{ W_{1},\dots,W_{d},V_{i_1},\dots,V_{i_{k-d}}\right\} $. Then
\[
\left|\mathcal{W}^{\left(d\right)}\right|\le\frac{\varepsilon\binom{t}{k}}{\sqrt{\varepsilon}\binom{t-d}{k-d}}\le\sqrt{\varepsilon}\binom{t}{d}.
\]
Now, consider any $i\le h$ and $J\subseteq \{1,\dots,h\}$, and suppose that the condition in (2) holds. Consider any $d$-set $\mathcal X$ of clusters in $\binom{\mathcal{P}_{i}}{d}\backslash\mathcal{W}^{\left(d\right)}$. We wish to estimate $\deg_{\mathcal{P}_{J}}\left(\mathcal X\right)$.

Let $E$ be the set of edges of $G'$ which have a vertex in each
of the $d$ clusters in $\mathcal X$, and $k-d$ vertices in
$P_{J}$. In order to estimate $\deg_{\mathcal{P}_{J}}\left(\mathcal X\right)$,
we count $|E|$ in two different ways. First, we break $E$ into subsets
depending on how they relate to the cluster graph $\mathcal{R}$.

Let $E_{\mathrm{\mathcal{R}}}$ be the set of $e\in E$ which ``arise
from $\mathcal{R}$'' in the sense that the vertices of $e$ come
from distinct clusters that form an edge in $\mathcal{R}$ containing
$\mathcal X$. By upper-uniformity, we have 
\[
|E_{\mathcal{R}}|\le(1+o(1))\deg_{\mathcal{P}_{J}}\left(\mathcal X\right)p(n/t)^{k}.
\]
Let $E_{\mathrm{irr}}$ be the subset of edges in $E$ arising from
irregular $k$-sets. By the choice of $\mathcal X\in\binom{\mathcal{P}_{i}}{d}\backslash\mathcal{W}^{\left(d\right)}$,
we have 
\[
|E_{\mathrm{irr}}|\le\sqrt{\varepsilon}\binom{t-d}{k-d}(1+o(1))p(n/t)^{k}.
\]
Let $E_{\tau}$ be the subset of edges in $E$ arising from $k$-sets
of clusters (containing $\mathcal X$) that are regular but whose
density is less than $\tau$ (and therefore do not appear in the cluster
graph). We have 
\[
|E_{\tau}|\le\tau\binom{t-d}{k-d}(1+o(1))p(n/t)^{k}.
\]
Let $E_{\mathrm{mul}}$ be the set of edges in $E$ which have multiple
vertices in the same cluster of $\mathcal{P}_{J}$, and let $E_{Z}$
be the set of edges in $E$ which involve a vertex not in a cluster
in $\mathcal{P}_{J}$ (because its cluster was not completely contained in any $P_j$). Simple double-counting arguments give
\[
|E_{\mathrm{mul}}|\leq\frac{\binom{k-d}{2}t^{k-d-1}}{\left(k-d\right)!}\left(1+o\left(1\right)\right)p\left(n/t\right)^{k},\quad|E_{Z}|\leq\frac{\left(k-d\right)(\varepsilon ht)t^{k-d-1}}{\left(k-d\right)!}\left(1+o\left(1\right)\right)p\left(n/t\right)^{k}.
\]
All in all, we obtain 
\begin{align*}
\left|E\right| & \leq|E_{\mathcal{R}}|+|E_{\mathrm{irr}}|+|E_{\tau}|+|E_{\mathrm{mul}}|+\left|E_{Z}\right|\\
 & \le\left(1+o\left(1\right)\right)p\left(\frac{n}{t}\right)^{k}\left(\deg_{\mathcal{P}_{J}}\left(\mathcal X\right)+t^{k-d}\left(\sqrt{\varepsilon}+\tau+(k-d)/t_{0}+\varepsilon h\right)\right).
\end{align*}
On the other hand, by the degree assumption in $G'$ and the fact
that $n_{J}\ge t_{J}\left(n/t\right)$, we have
\[
|E|\ge \left(\left(\frac{n}{t}\right)^d-o(n^d)\right)\delta p\binom{n_{J}-d}{k-d}\ge (1-o(1))p\delta \binom{t_{J}-d}{k-d}\left(\frac{n}{t}\right)^{k}.
\]
It follows that 
\[
\deg_{\mathcal{P}_{J}}\left(\mathcal X\right)\ge\delta \binom{t_{J}-d}{k-d}-\left(\sqrt{\varepsilon}+\tau+k/t_{0}+\varepsilon h\right)t^{k-d}.\tag*{\qedhere}
\]
\end{proof}

\subsection{A sparse embedding lemma}

One of the most powerful aspects of the sparse regularity method is
that, for a subgraph $G'$ of a typical outcome of a random graph,
if we find a substructure in the cluster graph (which is usually dense,
therefore comparatively easy to analyse), then a corresponding structure
must also exist in the original graph $G'$. For graphs, this was famously conjectured
to be true by Kohayakawa, \L uczak and R\"odl~\cite{KLR97}, and
was proved by Conlon, Gowers, Samotij and Schacht~\cite{CGSS14}.
We will need a generalisation to hypergraphs, which was already observed to hold in \cite{CGSS14} and appears explicitly as \cite[Theorem~4.12]{FK19}.
To state it we will need some definitions.
\begin{defn}
Consider a $k$-graph $H$ with vertex set $\left\{ 1,\dots,r\right\} $
and let $\mathcal{G}\left(H,n,m,p,\varepsilon\right)$ be the collection
of all $k$-graphs $G$ obtained in the following way. The vertex
set of $G$ is a disjoint union $V_{1}\cup\dots\cup V_{r}$ of sets
of size $n$. For each edge $\left\{ i_{1},\dots,i_{k}\right\} \in E\left(H\right)$,
we add to $G$ an $\left(\varepsilon,p\right)$-regular $k$-graph
with $m$ edges between $V_{i_{1}},\dots,V_{i_{k}}$. These are the
only edges of $G$.
\end{defn}

\begin{defn}
For $G\in\mathcal{G}\left(H,n,m,p,\varepsilon\right)$, let $\#_{H}(G)$
be the number of ``canonical copies'' of $H$ in $G$, meaning that
the copy of the vertex $i$ must come from $V_{i}$.
\end{defn}

\begin{defn}
\label{def:3-density}The $k$-\emph{density} $m_{k}\left(H\right)$
of a $k$-graph $H$ is defined as 
\[
m_{k}\left(H\right)=\max\left\{ \frac{e\left(H'\right)-1}{v\left(H'\right)-k}:H'\subseteq H\text{ with }v\left(H'\right)>k\right\}.
\]
\end{defn}

Now, our sparse embedding lemma is as follows.
\begin{thm}
\label{thm:KLR}For every linear $k$-graph $H$ and every $\tau>0$,
there exist $\varepsilon,\zeta>0$ with the following property. For
every $\kappa>0$, there is $C>0$ such that if $p\ge CN^{-1/m_{k}\left(H\right)}$,
then with probability $1-e^{-\Omega\left(N^{k}p\right)}$ the following
holds in $G\in{\operatorname H}^{k}\left(N,p\right)$. For every $n\ge\kappa N$,
$m\ge\tau pn^{k}$ and every subgraph $G'$ of $G$ in $\mathcal{G}\left(H,n,m,p,\varepsilon\right)$,
we have $\#_{H}(G')>\zeta p^{e(H)} n^{v(H)}$.
\end{thm}
Note that the condition $m\ge \tau pn^k$ is precisely the condition that the corresponding $G'\in \mathcal(H,n,m,p,\varepsilon)$ has density at least $\tau$.

\cref{thm:KLR} may be proved using the methods in \cite{CGSS14}. The necessary adaptations for the hypergraph setting are described in detail in \cite[Appendix~A]{FK19}.

\section{Concentration lemmas}\label{sec:concentration}

In this section we collect a number of basic facts about concentration of the edge distribution in random hypergraphs and random subsets of hypergraphs. First, we show that the upper-uniformity condition in the sparse regularity lemma is almost always satisfied in random hypergraphs.
\begin{lem}
Fix $k\in \mathbb N$, $D>1$ and $0<\eta<1$, and consider $G\in{\operatorname H}^{k}\left(n,p\right)$. Then $G$ is 
$\left(\eta,p,D\right)$-upper-uniform with probability at least $1-2^{k n} e^{-\Omega(n^kp)}$.
\end{lem}
\begin{proof}
Consider disjoint vertex sets $X_1,\dots,X_k$ each having size at least $\eta n$. Then $\E e(X_1,\dots,X_k)=p|X_1|\dots|X_k|=\Omega(n^k p)$, so by the Chernoff bound,
$$\Pr(e(X_1,\dots,X_k)\ge D p|X_1|\dots|X_k|)= \exp\left(-\Omega\left((D-1)^2 n^{k}p\right)\right)=e^{-\Omega(n^{k}p)}.$$
We can then take the union bound over all choices of $X_1,\dots,X_k$.
\end{proof}
The following corollary is immediate, and will be more convenient in practice.
\begin{cor}
\label{lem:random-upper-uniform}
Fix $k\in \mathbb N$, let $p=\omega(n^{1-k}\log n)$, and consider $G\in{\operatorname H}^{k}\left(n,p\right)$. Then $G$ is 
$\left(o(1),p,1+o(1)\right)$-upper-uniform with probability at least $1-e^{-\omega(n\log n)}$.
\end{cor}

Next, the following lemma shows that in random hypergraphs all vertices have about the expected degree into any large enough set.
\begin{lem}
\label{lem:delete-vertices-degree-ok}Fix $\lambda>0$ and $k\geq 3$. Then there is $C>0$ such that if $p\ge C n^{2-k}$, a.a.s.\ $G\sim{\operatorname H}^{k}\left(n,p\right)$ has the following property. For every vertex $w$ and every set $X\subseteq V(G)$ of at most $\lambda n$ vertices, there are at most $2\left(\lambda n\right)p\binom{n-2}{k-2}$ edges in $G$ containing $w$ and a vertex of $X$.
\end{lem}
\begin{proof}
For any vertex $w$ and any set $X$ of at most $\lambda n$ vertices, the expected number of edges containing both $w$ and
a vertex of $X$ is at most $\left(\lambda n\right)p\binom{n-2}{k-2}$.
Then the desired result follows from the Chernoff bound and the union
bound over at most $2^n$ choices of $X$.
\end{proof}

We also need a slightly more sophisticated version of \cref{lem:delete-vertices-degree-ok} that works for general $d$-degrees and has quite a weak assumption on $p$, but only gives a conclusion for almost all $d$-sets of vertices. For vertex sets $S,X$ in a $k$-graph $G$, let $Z_G(S,X)$ be the number of edges $e\in E(G)$ that contain $S$ and have non-empty intersection with $X$.
\begin{lem} \label{lem:delete-vertices-degree-almost-ok}Fix $\lambda>0$ and positive integers $d< k$. For $p\ge \omega(n^{d-k})$,
a.a.s.\ $G\sim{\operatorname H}^{k}\left(n,p\right)$ has the following property.
For every subset $X\subseteq V(G)$ of size $|X|\leq \lambda n$, there are at most $o(n^d)$ different $d$-sets $S\subseteq V(G)$ such that $Z_G(S,X)>2\binom k d(\lambda n)p\binom{n-d-1}{k-d-1}$.
\end{lem}

\begin{proof}
Let $\binom {[k]} d$ be the collection of all $d$-subsets of $\{1,\dots,k\}$. Consider the collection of all $k$-sets of vertices (i.e., all possible edges in $G$). For each such $k$-set $e$, let $\binom e d$ be the collection of $d$-subsets of $e$, and fix a (bijective) labelling $\phi_e:\binom e d\to \binom {[k]}d$. For each $d$-subset $I\subseteq \{1,\dots,k\}$ and set $S$ of $d$ vertices, let $Z_G^I(S,X)$ be the number of edges $e\in E(G)$ which include $S$, have non-empty intersection with $X$, and have $\phi_e(S)=I$.

Fix a vertex set $X$ of size at most $\lambda n$, and let $\mathcal E_S^I$ be the event that $Z_G^I(S,X)>2(\lambda n)p\binom {n-d-1}{k-d-1}$. By the Chernoff bound applied to $Z_G^I(S,X)$, we have $\Pr(\mathcal E_S^I)=o(1)$ for each $S,I$. For fixed $I$, each of the events $\mathcal E_S^I$ (ranging over different $S$) are independent from each other, so, again using a Chernoff bound, with probability $1-e^{-\omega(n^d)}$, all but $o(n^d)$ of the events $\mathcal E_S^I$ are satisfied. Taking a union bound over $O(1)$ choices of $I$ shows that with probability $1-e^{-\omega(n^d)}$, for all but $o(n^d)$ of the sets of $d$ vertices $S$, all the events of the form $\mathcal E_S^I$ hold, meaning that $Z_G(S,X)\le 2\binom k d(\lambda n)p\binom{n-d-1}{k-d-1}$. The desired result then follows from a union bound over choices of $X$.
\end{proof}

We will also need the following lemma, showing that if we consider a high-degree spanning subgraph of a typical outcome of a random hypergraph, then random subsets are likely to
inherit minimum degree properties.
\begin{lem}
\label{lem:degree-concentration}
Fix positive integers $d<k$, and fix any $0\le \mu\le 1$ and $0<\gamma,\sigma\le 1$. Then for any $p\ge n^{d-k}\log^{3} n$, a.a.s\ $G\sim{\operatorname H}^{k}\left(n,p\right)$
has the following property. Consider a spanning subgraph $G'\subseteq G$
with minimum $d$-degree at least $\left(\mu+\gamma\right)p\binom{n-d}{k-d}$,
and let $Y$ be a uniformly random subset of $\sigma n$ vertices
of $G'$. Then a.a.s.\ every $d$-set of vertices has degree at least
$\left(\mu+\gamma/2\right)p\binom{\sigma n-d}{k-d}$ into $Y$.
\end{lem}

\begin{proof}
For a $d$-set of vertices $A$ and $1\le t\le k-d$,
define $p_{t}(A)$ to be the number of pairs of edges $(e,f)\in G$
such that $e$ and $f$ both include $A$, and $\left|e\cup f\right|=2(k-d)-t$,
and define the random variable
\[
\bar{\Delta}_{A}=\sum_{t=1}^{k-d}p_{t}\left(A\right)\sigma^{2\left(k-d\right)-t}.
\]
For each $(k-d)$-set of vertices $e\subseteq V\left(G\right)$, let
$\xi_{e}$ be the indicator random variable for the event $e\cup A\in E\left(G\right)$.
Note that $\bar{\Delta}_A$ is a quadratic polynomial in the $\xi_{e}$,
and has expectation $E_{0}:=\E\bar{\Delta}_{A}=O\left(n^{2\left(k-d\right)-1}p^{2}\right)$.
For any $(k-d)$-sets of vertices $e,f\subseteq V\left(G\right)$, we
compute the expected partial derivatives
\[
\E\frac{\partial\bar{\Delta}_{A}}{\partial\xi_{e}}=O\left(n^{k-d-1}p\right),\quad\E\frac{\partial^{2}\bar{\Delta}_{A}}{\partial\xi_{e}\partial\xi_{f}}=O\left(1\right).
\]
Let $E_{1}$ be the common value of the $\E\left[\partial\bar{\Delta}_{A}/\partial\xi_{e}\right]$
and let $E_{2}$ be the maximum value of the $\E\left[\partial^{2}\bar{\Delta}_{A}/\partial\xi_{e}\partial\xi_{f}\right]$.
Let $E_{\ge0}=\max\left\{ E_{0},E_{1},E_{2}\right\} $ and let $E_{\ge1}=\max\left\{ E_{1},E_{2}\right\} $.
Given our assumption $p\ge n^{d-k}$ we have
\[
E_{\ge0}=O\left(1+n^{2\left(k-d\right)-1}p^{2}\right),\quad E_{\ge1}=O\left(1+n^{k-d-1}p\right).
\]

Applying a Kim--Vu-type polynomial concentration inequality (for example \cite[Theorem~1.36]{TV06}, with $\lambda=\Omega\left((t/\sqrt{E_{\ge 0}E_{\ge 1}})^{1/(2-1/2)}\right)$), we see that for all $t\ge0$,
\[
\Pr\left(\bar{\Delta}_{A}\ge E_{0}+t\right)\le\exp\left(-C_2\left(\frac{t}{\sqrt{E_{\ge0}E_{\ge1}}}\right)^{1/(2-1/2)}+(2-1)\log\binom n{k-d}\right)
\]
for some constant $C_2>0$.
It follows that
\begin{equation}
\bar{\Delta}_{A}\le\left(1+n^{2\left(k-d\right)-1}p^{2}\right)\log^{3}n\label{eq:delta}
\end{equation}
with probability $1-n^{\omega\left(1\right)}$. By a union bound,
a.a.s.\ this holds for all $d$-sets $A$. This is the only property of $G$
we require.

So, fix an outcome of $G$ satisfying \cref{eq:delta} for all $A$, and fix a spanning subgraph $G'\subseteq G$ with minimum $d$-degree at least
$\left(\mu+\gamma\right)p\binom{n-d}{k-d}$. Instead of considering
a uniformly random set of $\sigma n$ vertices, we consider a closely
related ``binomial'' random vertex subset $Y'$ obtained by including
each vertex with probability $\sigma':=\sigma-n^{-2/3}$ independently. This suffices because $Y'$ can a.a.s.\ be coupled as a subset of a uniformly random set of size $\sigma n$ (note that the standard deviation of the size of $Y'$ is $O(\sqrt n)$).

Let $A$ be a $d$-set of vertices, and let $\Gamma(A)$ be the link
$(k-d)$-graph of $A$ with respect to $G'$. Then $\deg_{Y'}(A)$
is the number of $e\in\Gamma(A)$ that are subsets of $Y'$, so it has
expected value $|\Gamma(A)|(\sigma')^{k-d}\ge\left(\mu+\gamma\right)p\binom{\sigma' n-d}{k-d}$.
By Janson's inequality (see \cite[Theorem~2.14]{JLR00}), our assumption on $p$, and \cref{eq:delta}, we have
\[
\Pr\left(\deg_{Y'}(A)\le\left(\mu+\gamma/2\right)p\binom{\sigma' n-d}{k-d}\right)\le\exp\left(-\frac{\left((\gamma/2)p\binom{\sigma' n-d}{k-d}\right)^{2}}{2\bar{\Delta}_{A}}\right)=o(n^{-d}).
\]
Then, take the union bound over all $A$.
\end{proof}

Finally, we also need the following ``almost-all'' version of \cref{lem:degree-concentration}.
\begin{lem}
\label{lem:degree-concentration-weak}
Fix positive integers $d<k$, and fix any $0\le \mu\le 1$ and $0<\gamma,\sigma\le 1$. Then for any $p=\omega(n^{d-k})$, a.a.s\ $G\sim{\operatorname H}^{k}\left(n,p\right)$
has the following property. Consider a spanning subgraph $G'\subseteq G$
such that all but at most $o(n^d)$ of the $d$-sets of vertices have degree at least $\left(\mu+\gamma\right)p\binom{n-d}{k-d}$,
and let $Y$ be a uniformly random subset of $\sigma n$ vertices
of $G'$. Then a.a.s.\ all but $o(n^d)$ of the $d$-sets of vertices in $G'$ have degree at least
$\left(\mu+\gamma/2\right)p\binom{\sigma n-d}{k-d}$ into $Y$.
\end{lem}

\begin{proof}
Let $f=(n^{d-k}/p)^{1/4}=o(1)$. First, recall the definition of $\bar\Delta_A$ from \cref{lem:degree-concentration}. By Markov's inequality, with probability at least $1-f=1-o(1)$ all but at most $f n^d=o(n^d)$ of the $d$-sets of vertices $A$ in $G'$ have $\bar \Delta_A\le \E \bar \Delta_A/f=O(n^{2(k-d)-1}p^2/f)$ and have degree at least $(\mu+\gamma)p\binom{n-d}{k-d}$ (in which case say $A$ is \emph{good}). So, fix an outcome of $G$ satisfying this property.

Now, let $A$ be a good $d$-set of vertices, and let $Y'$ be as in \cref{lem:degree-concentration}. By the same calculation as in \cref{lem:degree-concentration} (using Janson's inequality), we have
\[
\Pr\left(\deg_{Y'}(A)\le\left(\mu+\gamma/2\right)p\binom{\sigma n-d}{k-d}\right)=o(1),
\]
so by Markov's inequality, a.a.s.\ only $o(n^d)$ good $d$-sets fail to satisfy the degree condition (and there are only $o(n^d)$ non-good $d$-sets).
\end{proof}

\section{Almost-perfect matchings}
\label{sec:almost-perfect}

The easier part of the proof of \cref{thm:resilience} is to show that a degree condition implies the existence of \emph{almost}-perfect matchings, as follows.

\begin{lem}
\label{lem:almost-perfect}Fix positive integers $d<k$ and consider
any $\gamma>0$. Suppose $p=\omega\left(n^{1-k}\log n\right)$. Then, with
probability $1-e^{-\Omega\left(pn^{k}\right)}$, $G\sim{\operatorname H}^{k}\left(n,p\right)$
has the following property. In every spanning subgraph $G'\subseteq G$ for which all but $o(n^d)$ of the $d$-sets of vertices have degree at least $\left(\mu_{d}\left(k\right)+\gamma\right)p\binom{n-d}{k-d}$, there is a matching covering all but $o\left(n\right)$ vertices.
\end{lem}


We will prove \cref{lem:almost-perfect} with the sparse regularity lemma. It will be important that an almost-perfect matching in the cluster hypergraph $\mathcal R$ can be translated to an almost-perfect matching in the original hypergraph $G$. This will be deduced from the following lemma.

\begin{lem}\label{lem:regular-almost-perfect}
Consider an $(\varepsilon,p)$-regular $k$-partite $k$-graph $G$ with parts $V_1,\dots,V_k$ of the same size $m$, and density at least $2\varepsilon p$. Then $G$ has a matching of size at least $(1-\varepsilon)m$.
\end{lem}
\begin{proof}
By the definition of $(\varepsilon,p)$-regularity, for any choice of $V'_1\subseteq V_1,\ldots, V_k'\subseteq V_k$, satisfying $|V'_i|\geq \varepsilon m$ for all $i$, we have
$$|d(V'_1,\ldots,V'_k)-d(V_1,\ldots,V_k)|\leq \varepsilon p,$$ 
so $d(V'_1,\ldots,V'_k)\ge \varepsilon p$, meaning that the sets $V'_1,\ldots,V'_k$ span at least one edge.

Now, suppose for the purpose of contradiction that there is no matching of size $(1-\varepsilon)m$ in $G$, and consider a maximum matching $M$. Since $|M|\le (1-\varepsilon)m$, the set of uncovered vertices $V_i'$ in each $V_i$ has size at least $\varepsilon m$, so $V'_1,\ldots,V'_k$ span at least one edge, by the above discussion. But then this edge can be used to extend $M$ to a larger matching, contradicting maximality.
\end{proof}

Now we prove \cref{lem:almost-perfect}.

\begin{proof}[Proof of \cref{lem:almost-perfect}]
Choose large $Q$, divisible by $k$, such that $m_{d}\left(k,Q\right)/\binom{Q-d}{k-d}\le \mu_{d}\left(k\right)+\gamma/4$. Also, choose some small $\varepsilon>0$ and let $\lambda=\binom{Q}{d}\left(\sqrt{\varepsilon}+e^{-c\eta^{2}Q}\right)$ be as in \cref{lem:random-subset-degrees} (taking $\eta=\gamma/2$ and $\delta=\sqrt \varepsilon$). Choosing large $Q$, and $\varepsilon$ small relative to $Q$, we can make $\lambda$ arbitrarily small.

Since $p=\omega\left(n^{1-k}\log n\right)$, by \cref{lem:random-upper-uniform} $G$ is a.a.s.\ $\left(o\left(1\right),p,1+o\left(1\right)\right)$-upper-uniform. So, by \cref{lem:transfer-cluster-graph-sparse}, if we apply the sparse regularity
lemma (\cref{lem:sparse-regularity}) to $G'$ with small $\varepsilon$ and
large $t_{0}$, we obtain a cluster $k$-graph $\mathcal{R}$ such
that all but $\sqrt{\varepsilon}\binom{t}{d}$ of the $d$-sets of
clusters have degree at least $\left(\mu_{d}\left(k\right)+\gamma/2\right)\binom{n-d}{k-d}$.

Now, we randomly partition the $t$ clusters into $t/Q$ subsets of
size $Q$. By \cref{lem:random-subset-degrees}, with positive probability
all but a $\lambda$-fraction of the subsets have minimum
degree at least $\left(\mu_{d}\left(k\right)+\gamma/4\right)\binom{Q-d}{k-d}$.
By our choice of $Q$, each of these $Q$-subsets of clusters $\mathcal{S}$
has the property that $\mathcal{R}\left[\mathcal{S}\right]$ has a
perfect matching, and we can combine these to find a matching covering
all but $\lambda t$ of the $t$ vertices of the cluster
graph.

Each edge of this matching corresponds to a $k$-tuple of clusters
$\left(V_{i_{1}},\dots,V_{i_{k}}\right)$ in which we can
find a matching with $\left(1-\varepsilon\right)\left(n/t\right)$
vertices, by \cref{lem:regular-almost-perfect}. We can combine these matchings to get a matching in $G'$ covering at least
$(1-\varepsilon-\lambda)n$ vertices. Since $\varepsilon$ and $\lambda$ could have
been arbitrarily small, this implies that we can find a matching covering
all but $o\left(n\right)$ vertices.
\end{proof}

\section{Sparse absorption}
\label{sec:sparse-absorption}

The most challenging part of the proof of \cref{thm:resilience} is to prove a suitable
sparse analogue of the strong absorbing lemma. Specifically, the lemma
we will prove is as follows.
\begin{lem}
\label{lem:sparse-absorbing-lemma}Fix positive integers $d<k$, and some $\gamma>0$. There are $\lambda,C>0$ such that the following holds. For $p$ satisfying $p\ge\max\{ n^{-k/2+\gamma}, Cn^{2-k}\}$, a.a.s.\ $G\sim{\operatorname H}^{k}\left(n,p\right)$ has the following property.
For any spanning subgraph $G'$ of $G$ with minimum $d$-degree at least
$\left(\mu_{d}\left(k\right)+\gamma\right)p\binom{n-d}{k-d}$, there
is a set $X\subseteq V\left(G'\right)$ such that
\begin{enumerate}
\item [(i)]$\left|X\right|\le (\gamma/2)^k n$, and
\item [(ii)]for every set $W\subseteq V\left(G\right)\backslash X$ of
at most $\lambda n$ vertices, there is a matching in $G'$ covering
exactly the vertices of $X\cup W$ (provided $|X\cup W|$ is divisible by $k$).
\end{enumerate}
\end{lem}

Most of the rest of the paper will be devoted to proving \cref{lem:sparse-absorbing-lemma}, but first we  give the simple deduction of \cref{thm:resilience} from \cref{lem:delete-vertices-degree-almost-ok,lem:almost-perfect,lem:sparse-absorbing-lemma}.

\begin{proof}[Proof of \cref{thm:resilience}]
We may assume that $\gamma$ is small relative to $d$ and $k$ (the lemma statement only gets stronger as we decrease $\gamma$). Let $G\sim{\operatorname H}^{k}\left(n,p\right)$ and let $G'\subseteq G$ be a spanning subgraph with $\delta_d(G')\ge (\mu_d(k)+\gamma)p\binom {n-d}{k-d}$. By \cref{lem:sparse-absorbing-lemma}, a.a.s.\ $G'$ has an absorbing
subset $X$ of size at most $(\gamma/2)^k n$. Let $n'=n-\left|X\right|$
and observe that since $X$ is so small, we have by \cref{lem:delete-vertices-degree-almost-ok} that all but $o(n^d)$ $d$-subsets in $G'-X$ have degree at least $\left(\mu_{d}\left(k\right)+\gamma/2\right)\binom{n'-d}{k-d}p$ in $G'-X$. 

By \cref{lem:almost-perfect}, it follows that $G'-X$ has a matching 
covering all vertices of $V\left(G'\right)\backslash X$ except a
set $W$ of size $o\left(n\right)$. By the defining property of the
absorbing set $X$, it follows that $G'$ has a perfect matching.
\end{proof}

The crucial idea for the proof of \cref{lem:sparse-absorbing-lemma}
is to find many small subgraphs called \emph{absorbers}, which can each contribute to a matching in two different ways.
\begin{defn}[absorbers]
\label{def:absorber}An \emph{absorber} rooted on a $k$-tuple
of vertices $\left(x_{1},\dots,x_{k}\right)$ is a $k$-graph
on some set of vertices including $x_{1},\dots,x_{k}$, whose edges can be partitioned into:
\begin{itemize}
\item a perfect matching, in which each of $x_{1},\dots,x_{k}$ are in a
unique edge (the \emph{covering matching}), and;
\item a matching covering all vertices except $x_{1},\dots,x_{k}$ (the
\emph{non-covering matching}).
\end{itemize}
\end{defn}
We call $x_{1},\dots,x_{k}$ the \emph{rooted vertices}, and the $k$ edges containing each of $x_{1},\dots,x_{k}$ are called
\emph{rooted edges}.
The \emph{order} of the absorber is its number of vertices other
than $x_{1},\dots,x_{k}$.

Absorbers are the basic building blocks for a larger ``absorbing structure'', whose vertex set we will take as the set $X$ in \cref{lem:sparse-absorbing-lemma}. The relative positions of the absorbers in this
structure will be determined by a ``template'' with a ``resilient
matching'' property as will be described in the next few lemmas. 

\begin{lem}
\label{lem:resilient-template}For fixed $k\in{\mathbb N}$ there is $L>0$ such that the following holds.
For any sufficiently large $r$, there exists a $k$-graph $T$ with
at most $Lr$ vertices, at most $Lr$ edges, and an identified
set $Z$ of $r$ vertices, such that if we remove fewer than $r/2$
vertices from $Z$, the resulting hypergraph has a perfect matching (provided its number of vertices is divisible by $k$). We call $T$ a \emph{resilient template}
and we call $Z$ its \emph{flexible set}. We say $r$ is the \emph{order} of the resilient template.
\end{lem}

We defer the proof of \cref{lem:resilient-template} to \cref{subsec:construct-template}. It is a simple reduction from a random graph construction due to Montgomery~\cite{Mon14,Mon19}. We will want to arrange absorbers in the positions prescribed by a
resilient template, as follows.
\begin{defn}
An \emph{$\left(r,Q\right)$-absorbing structure} is a $k$-graph
$H$ of the following form. Consider an order-$r$ resilient
template $T$ and put externally vertex-disjoint absorbers of order at most $Q$ on each edge of $T$ (that is to say, the absorbers intersect only at their root vertices). We stress that the edges of $T$ are not actually present in $H$, they just describe the relative positions of the absorbers. See \cref{fig:absorbing-structure}.
\end{defn}

\begin{figure}[h]
\begin{center}
\includegraphics{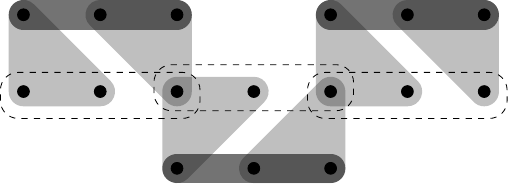}
\end{center}

\caption{\label{fig:absorbing-structure}A cartoon of an absorbing structure. The dashed bubbles indicate template edges, which are not actually a part of the absorbing structure.}

\end{figure}

An absorbing structure
$H$ has the same crucial property as the resilient template $T$
that defines it: if we remove fewer than half of the vertices of the flexible
set $Z$ then what remains of $H$ has a perfect matching. Indeed,
after this removal we can find a perfect matching $M$ of $T$, then
our perfect matching of $H$ can be comprised of the covering matching
of the absorber on each edge of $M$ and the non-covering matching
for the absorber on each other edge of $T$.

We will want to find an absorbing structure $H$ whose flexible set $Z$ has a certain ``richness'' property: essentially, we will want $Z$ to have the property that for all small sets $W$ disjoint from $H$, there is a matching covering $W$ and a small portion of $Z$ (and no other vertices). Sets $Z$ with this property always exist in the setting of \cref{lem:sparse-absorbing-lemma}, as follows.

\begin{lem}
\label{lem:rich-set}Fix positive integers $1\leq d<k$, and fix $\rho,\delta>0$. Then there is $\lambda>0$
such that the following holds. For $p\ge n^{1-k}\log^3 n$, a.a.s.\ $G\sim{\operatorname H}^{k}\left(n,p\right)$ has the
following property. Consider a spanning subgraph $G'\subseteq G$
with $\delta_d(G')\ge \delta p\binom{n-d}{k-d}$ for some $d\ge 1$. Then there is a set $Z$ of $\rho n$ vertices such that for any $W\subseteq V(G)\setminus Z$ with $|W|=\lambda n$, there is a matching $M$ in $G'$ covering all vertices in $W$, each edge of which contains one vertex of $W$ and $k-1$ vertices of $Z$.
\end{lem}

The proof of \cref{lem:rich-set} is not too difficult, but to avoid interrupting the flow of this section we defer its proof to \cref{subsec:rich-set}. Briefly, the idea is to show that a random set $Z$ typically does the job, using some concentration inequalities and a hypergraph matching criterion due to Aharoni and Haxell.

Having found a rich set $Z$ as guaranteed by \cref{lem:rich-set}, we need to show that $G'$ has an absorbing structure with $Z$ as its flexible set. We will greedily construct such an absorbing structure using the following lemma, which says that absorbers can be found rooted on any triple of vertices, even if a few vertices are
``forbidden''.
\begin{lem}
\label{lem:find-absorbers}Fix positive integers $1\leq d<k$ and some $\gamma>0$.
There is $Q\in{\mathbb N}$ and $C,\sigma>0$ such that the
following holds. For $p$ satisfying $p\ge \max\{n^{-k/2+\gamma},C n^{2-k}\}$, a.a.s.\ $G\sim{\operatorname H}^{k}\left(n,p\right)$ has the following property.
Every spanning subgraph $G'$ of $G$ with minimum $d$-degree at least
$\left(\mu_{d}\left(k\right)+\gamma\right)p\binom{n-d}{k-d}$ has
an absorber of order at most $Q$ rooted on any $k$-tuple of vertices, even
after deleting $\sigma n$ other vertices of $G'$.
\end{lem}
The proof of \cref{lem:find-absorbers} is quite involved, and contains the most interesting new ideas in this paper. We defer it to \cref{sec:absorbers}. Finally, we deduce \cref{lem:sparse-absorbing-lemma}.

\begin{proof}[Proof of \cref{lem:sparse-absorbing-lemma}]
Let $G'\subseteq G$ be a spanning subgraph with $\delta_d(G')\ge\left(\mu_{d}\left(k\right)+\gamma\right)p\binom{n-d}{k-d}$. Let $L$ be as in \cref{lem:resilient-template} and let $\sigma,Q$ be as in \cref{lem:find-absorbers}. We may assume $\sigma\le (\gamma/2)^k$. Choose small $\rho$ with $\rho(Q+1)L \le \sigma$.

By \cref{lem:rich-set}, a.a.s.\ we can find a $\rho n$-vertex ``rich''
set $Z$ having the property that for some small $\lambda>0$, every set of $\lambda n$ vertices can be ``matched into'' $Z$. We may assume $\lambda<\rho/(2k)$. By \cref{lem:resilient-template}
there is an order-$\rho n$ resilient template
$T$. Now, an absorbing structure on $T$ would have at most $QL\rho n+L\rho n$ vertices, and since $\rho(Q+1)L \le \sigma$, we can use \cref{lem:find-absorbers} to greedily build a $\left(\rho n,Q\right)$-absorbing structure $H\subseteq G'$ on the template $T$, having at most $(\gamma/2)^k n$ vertices,
with $Z$ as the flexible set. Let $X=V\left(H\right)$.

We now claim that $X$ satisfies the assumptions of the lemma. First, by the assumption $\sigma\le (\gamma/2)^k$, it has size at most $(\gamma/2)^k n$. For the second property, consider any $W\subseteq V\left(G'\right)\backslash X$ with
size at most $\lambda n$, such that $W\cup X$ is divisible by $k$. By the defining property of our rich set $Z$, we can find a matching $M_1$ in $G'$ covering $W$ and $(k-1)|W|\le (k-1)\lambda n<(\rho/2) n$ vertices of $X$. By the special property of our absorbing structure $H$, we can then find a matching $M_2$ covering all the remaining vertices in $G[X\cup W]$, and then $M_1\cup M_2$ is the desired matching.
\end{proof}

\subsection{Constructing an absorbing template}\label{subsec:construct-template}

In this subsection we prove \cref{lem:resilient-template}. We will
build our desired $k$-graph via a simple transformation from a bipartite
graph with certain properties. The following lemma was proved by Montgomery,
and appears as \cite[Lemma~10.7]{Mon19}. We write $\sqcup$ to indicate that
a union of sets is disjoint.

\begin{lem}
\label{lem:montgomery}
For any sufficiently large $s$, there exists a bipartite
graph $R$ with vertex parts $X$ and $Y\sqcup Z$, with $|X|=3s$,
$|Y|=|Z|=2s$, and maximum degree $100$, such that if we remove any
$s$ vertices from $Z$, the resulting bipartite graph has a perfect
matching.
\end{lem}

From \cref{lem:montgomery} we can deduce the following lemma (this is a $k$-uniform
version of \cite[Lemma 5.2]{Kwa16}).

\begin{lem}
\label{lem:montgomery-k}
For any sufficiently large $s$, there exists a $k$-partite
$k$-graph $S$ with vertex parts $X_{1},\dots,X_{k-1}$ and $Y\sqcup Z$,
with $|X_{1}|=\dots|X_{k-1}|=3s$, $|Y|,|Z|=2s$, and maximum degree
$100$, such that if we remove any $s$ vertices from $Z$, the resulting
$k$-graph has a perfect matching (provided its number of vertices
is even).
\end{lem}

\begin{proof}
Consider the bipartite graph $R$ from \cref{lem:montgomery} on the vertex
set $X\sqcup\left(Y\sqcup Z\right)$, and let $X_{k-1}=X$. Obtain a $k$-partite
graph $R'$ by adding sets $X_{1},\dots,X_{k-2}$ each having $\left|X\right|$
new vertices, and for each $1\le i<k-1$ putting an arbitrary perfect
matching between $X_{i}$ and $X_{i+1}$. Now, our $k$-partite $k$-graph
$S$ has the same vertex set as $R'$, and an edge for every $k$-vertex
path running through $X_{1},\dots,X_{k-1},Y\sqcup Z$ (call such paths
\emph{special paths}). Note that an edge in $R$ can be uniquely extended
to a special path in $R'$. Moreover, a matching in $R$ can always
be uniquely extended to a vertex-disjoint union of special paths in
$R'$.
\end{proof}

We also need the following simple lemma showing that there are sparse hypergraphs with no large independent sets.

\begin{lem}
\label{lem:random-independent-set}
For any $k$ there is some $K\in\mathbb{N}$ such that
the following holds. For sufficiently large $r$ there is a $k$-graph
$G$ with $r$ vertices and at most $Kr$ edges, with no independent
set of size $r/2$.
\end{lem}
\begin{proof}
Consider a random $k$-graph $G\sim{\operatorname H}^{k}\left(r,p\right)$ for $p=\left(K/2\right)r/\binom{r}{k}$.
For a set of $r/2$ vertices, the probability that there are no edges in that set is $(1-p)^{\binom{r/2}{k}}=e^{-\Omega(Kr)}$, so for large
$K$, the union bound shows that a.a.s.\ every
set of size $r/2$ induces at least one edge. Also, the Chernoff bound
shows that a.a.s.\ $G$ has at most $Kn$ edges.
\end{proof}

We are now ready to prove \cref{lem:resilient-template}.

\begin{proof}[Proof of \cref{lem:resilient-template}]
Start with the $k$-graph $S$ from \cref{lem:montgomery-k}, with $s=\ceil{r/2}$, and delete at most one vertex from $Z$ to make it have size $r$. Then, consider an $r$-vertex $k$-graph $G$ as in \cref{lem:random-independent-set} (which exists as
long as $r$ is large enough), and place $G$ on the vertex set $Z$.
Let $T$ be the resulting $k$-graph. It has at most $(k-1)(3s)+2s+2s$ vertices and at most $100(4s)+K(2s)$ edges.

Now, consider any set $W$ of fewer than $r/2$ vertices of $Z$,
such that the number of vertices in $T-W$ is divisible by $k$. By
the defining property of $G$, we can greedily build a matching $M_{1}$
in $T\left[Z\backslash W\right]$ covering all but $s$ vertices,
and then by the defining property of $S$, there is a perfect matching
$M_{2}$ in $T-\left(W\cup V\left(M_{1}\right)\right)$. Then $M_{1}\cup M_{2}$
is the desired perfect matching of $T-W$.
\end{proof}

\subsection{Finding a rich set of vertices}\label{subsec:rich-set}

In this subsection we prove \cref{lem:rich-set}. We will make use of the following Hall-type theorem for finding
large matchings in hypergraphs, due to Aharoni and Haxell~\cite{AH00}.
\begin{thm}
\label{thm:hyper-match}Let $\left\{ L_{1},\dots,L_{t}\right\} $
be a family of $k'$-uniform hypergraphs on the same vertex set. If,
for every $\mathcal{I}\subseteq\left\{ 1,\dots,t\right\} $, the hypergraph
$\bigcup_{i\in\mathcal{I}}L_{i}$ contains a matching of size greater
than $k'\left(\left|\mathcal{I}\right|-1\right)$, then there exists
a function $g:\left\{ 1,\dots,t\right\} \rightarrow\bigcup_{i=1}^{t}E\left(L_{i}\right)$
such that $g\left(i\right)\in E\left(L_{i}\right)$ and $g\left(i\right)\cap g\left(j\right)=\emptyset$
for $i\neq j$.
\end{thm}

To apply \cref{thm:hyper-match}, the following lemma will be useful, concerning the distribution of edges in random hypergraphs.

\begin{lem}
\label{lem:hyper-match-concentration}
Fix $k,q\in\mathbb{N}$ and $\lambda>0$, and suppose $p=\omega\left(n^{1-k}\log n\right)$.
Then a.a.s.\ $G\sim \operatorname{H}^{k}\left(n,p\right)$ has the
following property. For every pair of vertex sets $\mathcal{I},U$
with $\left|\mathcal{I}\right|\le\lambda n$ and $\left|U\right|\le q\left|\mathcal{I}\right|$,
there are at most $2\lambda qp\left|\mathcal I\right|n^{k-1}$ edges which contain
a vertex from $\mathcal{I}$ and a vertex from $U$.
\end{lem}

\begin{proof}
Fix $\mathcal{I},U$ as in the lemma statement. By the Chernoff
bound, the number of edges intersecting $\mathcal{I}$ and $U$ is
at most $2\lambda qp\left|\mathcal I\right|n^{k-1}$ , with probability $1-\exp\left(-\Omega\left(p\left|\mathcal{I}\right|n^{k-1}\right)\right)$.
So, by the union bound, the probability that the property in the lemma
statement fails is at most
\[
\sum_{i=1}^{\lambda n}n^{i}n^{qi}\exp\left(-\Omega\left(pin^{k-1}\right)\right)=n^{-\omega\left(1\right)}.\tag*{\qedhere}
\]
\end{proof}

We will also need the (very simple) fact that minimum $d$-degree assumptions are strongest when $d$ is large.

\begin{lem}
\label{lem:codegree-degree}
Let $G$ be an $n$-vertex $k$-graph. If $d\ge d'$ and $\delta_d(G)\ge \alpha \binom{n-d}{k-d}$ then $\delta_{d'}(G)\ge \alpha \binom{n-d'}{k-d'}$.
\end{lem}

\begin{proof}
Suppose that $\delta_d(G)\ge \alpha \binom{n-d}{k-d}$, and fix any subset $S$ of $d'$ vertices of $G$. By assumption, for every subset $X\subseteq V(G)\setminus S$ of size $d-d'$, the number of edges containing $X\cup S$ is at least $\alpha \binom{n-d}{k-d}$. Since each edge containing $S$ is being counted exactly $\binom{k-d'}{d-d'}$ times, we conclude that 
\[\delta_{d'}(G)\geq \frac{\binom{n-d'}{d-d'}\alpha \binom{n-d}{k-d}}{\binom{k-d'}{d-d'}}=\alpha \binom{n-d'}{k-d'}.\tag*{\qedhere}\]
\end{proof}

Now, we prove \cref{lem:rich-set}.

\begin{proof}[Proof of \cref{lem:rich-set}]
Fix an outcome of $G$ that satisfies the property in \cref{lem:hyper-match-concentration}, for $q=(k-1)^2$ and small $\lambda>0$ to be determined, and also satisfies the property in \cref{lem:degree-concentration}, for $\mu=\gamma=\delta/2$ and $d=1$. Consider a spanning subgraph $G'\subseteq G$ with minimum $1$-degree at least $\delta p\binom{n-1}{k-1}$. Let $Z$ be a set of $\rho n$ vertices such that every vertex outside $Z$ has degree at least $\left(\delta/2\right)p\binom{\rho n-1}{k-1}$ into $Z$ (by \cref{lem:codegree-degree,lem:degree-concentration}, almost every choice of $Z$ will do).

Now, consider any $W\subseteq V\left(G'\right)\backslash Z$ with
size $\lambda n$. For each $w\in W$ let $L_{w}$
be the link $\left(k-1\right)$-graph of $w$ into $Z$ (having an
edge $e\subseteq Z$ whenever $e\cup\left\{ w\right\} $ is an edge
of $G'$). We claim that if $\lambda$ is sufficiently small then for each $\mathcal{I}\subseteq W$, the $(k-1)$-graph $H_{\mathcal I}:=\bigcup_{w\in\mathcal{I}}L_{w}$ has a matching of size greater than $\left(k-1\right)\left(\left|\mathcal{I}\right|-1\right)$. The desired result will then follow from \cref{thm:hyper-match}. Actually, it will be convenient to view $H_{\mathcal I}$ as a multigraph (if an edge appears in multiple different $L_w$, among $w\in \mathcal I$, then we include that edge multiple times). This does not affect the existence of matchings, but the correspondence between edges of $H_{\mathcal I}$ and edges of $G'$ will be more natural.

To prove the claim, suppose for the purpose of contradiction that there is some $\mathcal{I}\subseteq W$ for which a maximum matching $M$ in $H_{\mathcal I}$ has size at most $\left(k-1\right)\left(\left|\mathcal{I}\right|-1\right)$ (meaning that it has at most $(k-1)^2(|\mathcal I|-1)$ vertices). By the degree condition defining $Z$, $H_{\mathcal I}$ has at least $|\mathcal I|(\delta/2)p\binom{\rho n-1}{k-1}$ edges. On the other hand, all edges of $H_{\mathcal I}$ intersect $V(M)$ by maximality, so by the property in \cref{lem:hyper-match-concentration}, $H_{\mathcal I}$ only has at most $2\lambda qp\left|\mathcal{I}\right| n^{k-1}$ edges. This is a contradiction if $\lambda$ is sufficiently small.
\end{proof}

\section{Finding absorbers}
\label{sec:absorbers}

Now we are finally ready to prove \cref{lem:find-absorbers}, showing that dense subgraphs of random graphs have absorbers and completing the proof of \cref{lem:sparse-absorbing-lemma}.

The main difficulty with finding absorbers is that they are \emph{rooted} objects. It is not enough to find an absorber floating somewhere in our graph (which we could easily do with the sparse embedding lemma); what we want is to find an absorber on a specific $k$-tuple of vertices. In order to achieve this, we define a contraction operation that reduces the task of finding a rooted absorber to the task of finding ``contracted absorbers'' in a much more flexible setting.

\begin{defn}[contractible absorbers]
\label{def:sub-absorber}A \emph{contractible absorber} rooted on a $k$-tuple of vertices $\left(x_{1},\dots,x_{k}\right)$
is an absorber obtained in the following way. Put $k$ disjoint
edges $e_{1}=\left\{ x_{1},y_{1}^{1},\dots,y_{1}^{k-1}\right\} ,\dots,e_{k}=\left\{ x_{k},y_{k}^{1},\dots,y_{k}^{k-1}\right\} $,
then for each $i$ put an externally vertex disjoint absorber
$H_{i}$ rooted on $\left\{ y_{1}^{i},\dots,y_{k}^{i}\right\}$ (we call each of these a \emph{sub-absorber}). Note that the edges in the non-covering matching of a contractible absorber come from the covering matchings of its constituent sub-absorbers, and the non-rooted edges in the covering matching of the contractible absorber come from the non-covering matchings of its sub-absorbers.

A \emph{contracted absorber} rooted at $\left(x_{1},\dots,x_{k}\right)$ is a hypergraph obtained as the union of $k-1$ absorbers (which we again call sub-absorbers) each rooted at $\left(x_{1},\dots,x_{k}\right)$, disjoint except for their rooted vertices. One can show that a contracted absorber is always itself an absorber, but we will not need this fact. The \emph{contraction} of a contractible absorber is the contracted absorber obtained by contracting each of its rooted edges to a single vertex. See \cref{fig:absorber}.
\end{defn}

\begin{figure}[h]
\begin{center}
\includegraphics{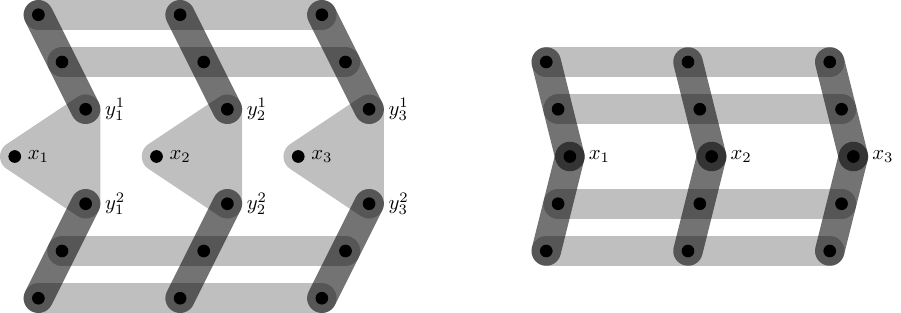}
\end{center}

\caption{\label{fig:absorber}An illustration of an order-18 contractible absorber rooted
on $\left(x_{1},x_{2},x_{3}\right)$, and its contraction. In this 3-uniform case the absorber
construction involves two sub-absorbers. The dark hyperedges are the non-covering
matching of the contractible absorber and the light hyperedges are the covering matching.}
\end{figure}

We will need our absorbers to satisfy a local sparsity condition to apply the sparse embedding lemma. We recall the definition of the girth of a hypergraph, and introduce a very closely related notion of local sparsity.
\begin{defn}\label{def:berge}
A (Berge) \emph{cycle} in a hypergraph is a sequence of edges $e_1,\dots,e_\ell$ such that there exist distinct vertices $v_1,\dots,v_\ell$ with $v_i\in e_i\cap e_{i+1}$ for all $i$ (where $e_{\ell+1}=e_1$). The \emph{length} of such a cycle is its number of edges $\ell$. The \emph{girth} of a hypergraph is the length of the shortest cycle it contains (if the hypergraph contains no cycle we say it has infinite girth, or is \emph{acyclic}). We say that an absorber rooted on $x_1,\dots,x_k$ is \emph{$K$-sparse} if it has girth at least $K$, even after adding the extra edge $\{x_1,\dots,x_k\}$ (that is to say, it has high girth and moreover the roots are ``far from each other'').
\end{defn}

(The fact that the roots are ``far from each other'' will allow us to ``glue together'' absorbers in various ways without worrying about the girth increasing). Recall the definition of the $k$-density $m_k(H)$ from \cref{def:3-density}.

\begin{lem}
\label{lem:contracted-absorber}Fix $\eta>0$ and $k\in{\mathbb N}$. There
is $K>0$ such that the following holds.
For any $k$-uniform contracted absorber $H$ with girth at least $K$, we have
$m_{k}\left(H\right)\le2/k+\eta$.
\end{lem}

The proof of \cref{lem:contracted-absorber} is basically just a calculation so we defer it to \cref{subsec:locally-sparse}.

We will also need the following lemma, showing that we can use the
definition of $\mu_{d}(k)$ itself to find locally sparse absorbers.
We will apply this lemma to a cluster $k$-graph obtained via the
sparse regularity lemma.

\begin{lem}
\label{lem:find-subabsorber}For any positive integers $d<k$ and any $\eta,K>0$, there are $\delta,M>0$ such that the following holds for sufficiently large $n$. Consider a $k$-graph
$G$ on $n$ vertices such that all but $\delta\binom{n}{d}$
$d$-sets of vertices have degree at least $\left(\mu_{d}\left(k\right)+\eta\right)\binom{n-d}{k-d}$.
Then for any vertices $x_{1},\dots,x_{k}$ which each have degree
at least $\eta\binom{n-1}{k-1}$, there is a $K$-sparse absorber with order
at most $M$, rooted on those vertices.
\end{lem}

Without the local sparsity condition it would be fairly easy to prove \cref{lem:find-subabsorber}, more or less by using the definition of the Dirac threshold twice. To deal with the sparseness condition we fix a locally sparse ``pattern'' (which is a high-girth bipartite graph), and use the definition of the Dirac threshold plus a random sampling trick to find an absorber ``in line with the pattern''. We defer the details to \cref{subsec:find-sparse-absorber}.

As previously mentioned, for the proof of \cref{lem:find-absorbers} we will use a contraction trick: we ``contract'' $G'$ to obtain a $k$-graph $G'(\mathcal F,\mathcal P)$, in such a way that if we can find contracted absorbers in $G'(\mathcal F,\mathcal P)$ satisfying certain properties, then these correspond to rooted contractible absorbers in $G'$. Before proving \cref{lem:find-absorbers} we define this ``contraction'' operation.

\begin{defn}
\label{def:contract-G}
Consider any $k$-graph $G$, consider a family $\mathcal{F}\subseteq V(G)^{k-1}$ of disjoint $\left(k-1\right)$-tuples, and consider a family $\mathcal{P}$ of disjoint sets $U_{1},\dots,U_{k-1}\subseteq V(G)$, such that the tuples in $\mathcal F$ and the sets in $\mathcal P$ do not share any vertices. Let $v(\mathcal P)$ be the total number of vertices in the sets in $\mathcal P$.

Let $G\left(\mathcal{F},\mathcal{P}\right)$ be the $k$-graph obtained as follows. Start with the $k$-graph $G\left[U_{1}\cup\dots\cup U_{k-1}\right]$, and for each $\boldsymbol v\in \mathcal{F}$ add a new vertex $w_{\boldsymbol v}$. For each tuple $\boldsymbol{v}=\left(v_{1},\dots,v_{k-1}\right)\in\mathcal{F}$, each $1\le j\le k-1$, and each $f\subseteq U_j$ such that $f\cup \{v_j\}$ is an edge of $G$, put an edge $f\cup \{w_{\boldsymbol v}\}$ in $G(\mathcal F,\mathcal P)$.
\end{defn}

One should visualise each $\boldsymbol v=(v_1,\dots,v_{k-1})\in \mathcal F$ as being ``contracted'' to a single vertex $w_{\boldsymbol v}$, and all edges involving $\boldsymbol v$ being deleted except those edges that contain some $v_j$ and have all their other vertices in the corresponding $U_j$. The reason for the edge deletion is to ensure that each edge in $G(\mathcal F,\mathcal P)$ corresponds to exactly one edge in $G$. So, if $G\sim{\operatorname H}^{k}(n,p)$ is a random $k$-graph, then $G(\mathcal F,\mathcal P)$ may be interpreted as a subgraph of a random graph with the same edge probability $p$, and we may apply the sparse embedding lemma to it. In our proof of \cref{lem:find-absorbers} we will choose $\mathcal F$ and $\mathcal P$ depending on our desired roots and the structure of $G'$.

Now we finally prove \cref{lem:find-absorbers}.
\begin{proof}[Proof of \cref{lem:find-absorbers}]
First note that we can assume $p=\Omega(n^{d-k}\log n)=\omega(n^{d-k})$, because otherwise $G$ itself a.a.s.\ has minimum $d$-degree zero and the lemma statement is vacuous. Also, we can assume that $\gamma$ is sufficiently small with respect to $k$ (the lemma statement only gets stronger as we decrease $\gamma$). Now, there are a number of constants in our proof that are defined in terms of each other (constants in the sense that they do not depend on $n$). First, let $\beta=1/(3k)$. Second, let $K$ be large enough to satisfy \cref{lem:contracted-absorber}, applied with $\eta=\gamma/2$. Third, let $M$ be large enough and $\delta>0$ be small enough to satisfy \cref{lem:find-subabsorber}, applied with $\eta=\gamma/4$ and the value of $K$ just defined. Then, $\sigma,\tau>0$ will be small relative to all constants defined so far (small enough to satisfy certain inequalities later in the proof), and we let $\alpha = \sigma /k^3$. Next, $\varepsilon>0$ will be very small, and $t_0$ very large, even compared to $\sigma$. Finally, $\kappa>0$ will be tiny compared to all other constants.

We record some properties that $G$ and each of the $G(\mathcal F,\mathcal P)$ a.a.s.\ satisfy.

\begin{claim*}
$G$ a.a.s.\ satisfies each of the following properties.
\begin{enumerate}
\item[(1)]{For each $\mathcal F,\mathcal P$ as in \cref{def:contract-G}, such that $|\mathcal F|=k\alpha n$, the $k$-graph $G\left(\mathcal{F},\mathcal{P}\right)$ satisfies the conclusion of \cref{thm:KLR} (the sparse embedding lemma), for embedding all graphs $H$ on at most $(k-1)(M+k)$ vertices which have $m_k(H)\le 2/k+\gamma/2$. (The other parameters $\tau,\kappa$ with which we apply the sparse embedding lemma are as described at the beginning of the proof, and the lemma gives us an upper bound on $\varepsilon$ in terms of $\tau$).}
\item[(2)]{Each $G\left(\mathcal{F},\mathcal{P}\right)$ as above is $\left(o\left(1\right),p,1-o\left(1\right)\right)$-upper-uniform.}
\item[(3)]{For any set $W'$ of $2\sigma n$ vertices and every vertex $x$, there are at most $4\left(\sigma n\right)p\binom{n-2}{k-2}$ edges in $G$ containing $x$ and a vertex of $W'$.}

\item[(4)]{For any spanning subgraph $G'\subseteq G$ with minimum $d$-degree
at least $\left(\mu_d(k)+\gamma\right)p\binom{n-d}{k-d}$, and any set $W'$ of at most $2\sigma n$ vertices, let $Y$ be a uniformly random subset of $\beta n$ vertices of $V(G')\setminus W'$. Then a.a.s\ all but $o(n^d)$ of the $d$-sets of vertices in $G'$ have $d$-degree at least $\left(\mu_d+\gamma/2\right)p\binom{\beta n-d}{k-d}$ into $Y$.}
\end{enumerate}
\end{claim*}
\begin{proof}
Observe that each $G\left(\mathcal{F},\mathcal{P}\right)$
has at least $|\mathcal F|=\Omega(n)$ vertices and it can be coupled as a subset of the binomial random
$k$-graph on its vertex set (with edge probability $p$). Indeed, for each possible edge $e$ of $G\left(\mathcal{F},\mathcal{P}\right)$, there is a single $k$-set $\phi(e)$ whose presence as an edge in $G$ determines whether $e$ is in $G\left(\mathcal{F},\mathcal{P}\right)$.

For (1) and (2), there are at most $2^{nk}n^{k\cdot k\alpha n}=\exp(O(n\log n))$ ways to choose $\mathcal F$ and $\mathcal P$. Since $p=\omega(n^{1-k}\log n)$, we may apply \cref{thm:KLR} and \cref{lem:random-upper-uniform} and take the union bound over all possibilities for $\mathcal F$ and $\mathcal P$. Note that the requirement on $p$ in \cref{thm:KLR} is that $p$ exceeds $n^{-1/m_k(H)}$ by a large constant, where $m_k(H)\le 2/k+\gamma/2$. This is certainly satisfied since we are assuming $p\ge n^{-k/2+\gamma}$, and that $\gamma$ is small with respect to $k$.

For (3) we simply apply \cref{lem:delete-vertices-degree-ok} with $\lambda=2\sigma$, and for (4) we apply \cref{lem:degree-concentration-weak} to $G'-W$ (recalling our assumption that $p=\omega(n^{d-k})$), after applying \cref{lem:delete-vertices-degree-almost-ok}.
\end{proof}
Consider an outcome of $G$ satisfying all the above properties (for the
rest of the proof, we can forget that $G$ is an instance of a random
graph and just work with these properties). Let $G'\subseteq G$ be a spanning subgraph with minimum $d$-degree at least $\left(\mu_{d}\left(k\right)+\gamma\right)p\binom{n-d}{k-d}$,
consider any set $W$ of $\sigma n$ vertices, and consider vertices $x_{1},\dots,x_{k}$
outside $W$. We will show that $G'-W$ has an absorber
rooted on $x_{1},\dots,x_{k}$ of order at most $(k-1)(M+k)$.

We will accomplish this by studying $G'(\mathcal F,\mathcal P)$ for certain $\mathcal F$ and $\mathcal P$. First, $\mathcal F$ will be defined in terms of the edges incident to $x_1,\dots,x_k$, using the following claim.
\begin{claim*}
For each $i$ let $\Gamma(x_i)$ be the link $\left(k-1\right)$-graph of $x_{i}$ with respect to $G'-W$. Then we can find matchings $M^{i}\subseteq E(\Gamma(x_i))$ of size $\alpha n$, such that no two of these matchings share a vertex.
\end{claim*}
\begin{proof}
First, ignoring the disjointness condition, note that (3) implies that each $\Gamma(x_i)$ has a matching $M^{i}_0$ of size $\sigma n/k$. Indeed, suppose a maximum matching in $\Gamma(x_i)$ were to have fewer than $\sigma n/k$ edges. Then the set of vertices of this matching would comprise a set $W'$ of fewer than $\sigma n$ vertices such that all the edges of $G'$ which contain $x$ intersect $W'\cup W$. But by the minimum degree condition on $G'$ (and \cref{lem:codegree-degree}) there are at least $(\mu_d(k)+\gamma)p\binom{n-1}{k-1}$ such edges, contradicting (3) for small $\sigma$.

We can then delete some edges from the $M^i_0$ to obtain the desired matchings $M^i$, recalling that $\alpha = \sigma /k^3$.
\end{proof}

Now we can define $\mathcal F$: for each $i$, arbitrarily order the vertices in each edge of $M^{i}$ to obtain
a collection $\mathcal{F}^{i}$ of $\alpha n$ disjoint $\left(k-1\right)$-tuples. Let $\mathcal{F}=\bigcup_{i=1}^{k}\mathcal{F}^{i}$, and let $V\left(\mathcal{F}\right)$ be the set of vertices in the
tuples in $\mathcal{F}$.

Next, we will choose the sets $U_1,\dots,U_{k-1}$ in $\mathcal P$ randomly, so that they satisfy a certain degree condition. This is encapsulated in the following claim.

\begin{claim*}
There are disjoint $\beta n$-vertex sets $U_1,\dots,U_{k-1}\subseteq V(G)\setminus(W\cup V(\mathcal F))$ such that all but $o(n^d)$ $d$-sets of vertices in $G'$ have degree at least $\left(\mu_{d}\left(k\right)+\gamma/3\right)\binom{\beta n-d}{k-d}$ into each $U_j$.
\end{claim*}
\begin{proof}
By (4), almost any choice of $U_{1},\dots,U_{k-1}\subseteq V(G)$ will do.
\end{proof}

Now, in $G'(\mathcal F,\mathcal P)$, let $X^{i}\subseteq V(G'(\mathcal F,\mathcal P))$ be the set of ``newly contracted''
vertices arising from tuples in $\mathcal{F}^{i}$, and let $X=\bigcup_{i=1}^{k}X^{i}$
be the set of all newly contracted vertices. Note that $G'\left(\mathcal{F},\mathcal{P}\right)$
is a subgraph of $G\left(\mathcal{F},\mathcal{P}\right)$. Our goal from now on is to prove the following claim.

\begin{claim*}
There is a set of vertices $y_{1}\in X^{1},\dots,y_{k}\in X^{k}$ such that for each $1\le j\le k-1$ there is an absorber of order at most $M$ in $G'(\mathcal F,\mathcal P)$ rooted at $y_1,\dots ,y_k$, whose other vertices lie entirely in $U_j\setminus(W\cup X)$.
\end{claim*}

Recalling the definition of $\mathcal F$ and $G(\mathcal F,\mathcal P)$, the absorbers in the above claim then form sub-absorbers for some contractible absorber in $G'$ of order at most $(k-1)M$, rooted at $x_1,\dots,x_k$. So, it suffices to prove the above claim to complete the proof of \cref{lem:find-absorbers}. Crucially, in our new goal, we have quite a lot of freedom to choose the roots $y_1,\dots,y_k$. This will allow us to use the sparse embedding lemma (that is, property (1)).

\begin{proof}[Proof of claim]
Recalling property (2), we apply our sparse regularity lemma (\cref{lem:sparse-regularity}) to
$G'\left(\mathcal{F},\mathcal{P}\right)$, with small $\varepsilon$
and large $t_{0}$. Apply \cref{lem:respect-partition} and \cref{lem:transfer-cluster-graph-partition}
with small threshold $\tau$, to obtain a $t$-vertex cluster
graph $\mathcal{R}$, with $t_{0}\le t$. Let $\mathcal{U}_{j}$ be the set of clusters (vertices of $\mathcal{R}$) contained
in each $U_{j}\setminus(W\cup X)$, let $\mathcal{X}^{i}$ be the set of clusters contained
in each $X^{i}$, and let $\mathcal W$ be the set of clusters contained in $W$.

Since $\sigma,\tau,\varepsilon$ are very small, and $t_0$ is very large, \cref{lem:respect-partition} and \cref{lem:transfer-cluster-graph-partition} ensure that for each $i$, $\mathcal X^i$ has about $\alpha t$ clusters, and almost all of those clusters have degree at least $\left(\mu_{d}\left(k\right)+\gamma/4\right)\binom{\beta t-1}{k-1}$ into each $\mathcal U_j$. Fix such a cluster $V^i$, for each $i$.

Now, $\mathcal R[\mathcal U_j\cup \{V^1,\dots,V^{k}\}]$ satisfies the assumptions for \cref{lem:find-subabsorber}, so it contains a $K$-sparse absorber in $\mathcal R$ of order at most $M$ rooted at $V^1,\dots, V^k$ whose other vertices (clusters) lie in $\mathcal U_j$. These absorbers form the sub-absorbers for a contracted absorber in $\mathcal R$, which has girth at least $K$ and therefore has $k$-density at most $2/k+\gamma/2$ by \cref{lem:contracted-absorber}. Now, by property (1), a canonical copy of the same contracted absorber exists in $G'(\mathcal F,\mathcal P)$ (in fact, there are many such copies). The sub-absorbers of this contracted absorber then satisfy the requirements of the claim.
\end{proof}

This completes the proof of \cref{lem:find-absorbers}.
\end{proof}

\subsection{Locally sparse absorbers have low $k$-density}
\label{subsec:locally-sparse}
In this section we prove \cref{lem:contracted-absorber}.
\begin{proof}[Proof of \cref{lem:contracted-absorber}]
Note that a cycle of length two corresponds to a pair of edges that intersect in more than one vertex (so a hypergraph has girth greater than two if and only if it is linear). The \emph{line graph} $L\left(G\right)$ of a linear $k$-graph $G$
has the edges of $G$ as vertices, with an edge $e_{1}e_{2}$ when
$e_{1}$ and $e_{2}$ are incident in $G$. Note that $v(L(G))=e(G)$ and $v(G)\ge ke(G)-e(L(G))$. Observe that a linear $k$-graph $G$ is acyclic if and only if its line graph is a forest, in which case $e(L(G))\le v(L(G))-1=e(G)-1$, so $v(G)\ge (k-1)e(G)+1$.

Now consider any subgraph $H'\subseteq H$
with $v\left(H'\right)>k$. First, if $H'$ is acyclic, then
\[
\frac{e\left(H'\right)-1}{v\left(H'\right)-k}\le\frac{e\left(H'\right)-1}{\left(k-1\right)e\left(H'\right)+1-k}=\frac{1}{k-1}\le2/k,
\]
by the above discussion. Otherwise, $H'$ has a cycle, which must have length at least $K$, so $v(H')\ge (k-1)K$. Note that every vertex of $H$ has degree at most 2, except $k$ vertices $x_{1},\dots,x_{k}$
which have degree $k-1$. So, $ke(H')$ can be bounded by $k(k-1)+2(v(H')-k)$, and
\[
\frac{e\left(H'\right)-1}{v\left(H'\right)-k}\le\frac{\left(k\left(k-1\right)+2(v(H')-k)\right)/k}{v(H')-k}\le\frac{k(k+1)/(K(k-1)-k)+2}{k}\le \frac{2}{k}+\eta
\]
for large $K$.
\end{proof}

\subsection{Locally sparse absorbers in dense hypergraphs}\label{subsec:find-sparse-absorber}

In this subsection we prove \cref{lem:find-subabsorber}. If we were to ignore the local sparseness condition it would be quite simple to find an absorber rooted on our desired vertices: recalling that an absorber essentially consists of two perfect matchings on some vertex set, we could simply apply the definition of the Dirac threshold twice, in appropriate subgraphs of $G$.

In order to deal with the sparseness condition, we fix a high-girth hypergraph $L$ (with edges of size about $\log n$), which will form a ``pattern'' for our absorber. We consider a random injection from $L$ into the vertex set of our graph $G$, so that (with positive probability) each of the edges of $L$ corresponds to a subgraph of $G$ with minimum $d$-degree exceeding its Dirac threshold, and therefore has a perfect matching. We will define $L$ in such a way that the union of these perfect matchings gives us an absorber with the desired properties.

It is convenient to deduce \cref{lem:find-subabsorber} from a slightly simpler lemma where no vertices of exceptionally low degree are allowed. To state this lemma we define a slight generalisation of the notion of an absorber.

\begin{defn}
An $r$\emph{-absorber} rooted at a $rk$-set of vertices
$y_{1},\dots,y_{rk}$ is a hypergraph which can be partitioned into
two matchings, one of which covers the entire vertex set and the other
of which covers every vertex except $y_{1},\dots,y_{rk}$. We say
an $r$-absorber is $K$-sparse if it has girth at least $K$, even
after adding an extra edge $\left\{ y_{1},\dots,y_{rk}\right\} $.
The point of this definition is that if we have roots $x_{1},\dots,x_{k}$
and we pick arbitrary disjoint edges $e_{1},\dots,e_{k}$ containing
the roots (whose other vertices are $y_{1},\dots,y_{\left(k-1\right)k}$,
say), then a $K$-sparse $\left(k-1\right)$-absorber rooted at $y_{1},\dots,y_{\left(k-1\right)k}$ (not containing the vertices $x_1,\dots,x_k$)
gives us a $K$-sparse absorber rooted at $x_{1},\dots,x_{k}$.
\end{defn}

Now, the key lemma is as follows.
\begin{lem}
\label{lem:find-q-absorber}
For any $\gamma>0$ and $r,k,K\in\mathbb{N}$ the following holds for sufficiently large $n$. For any $k$-graph
$G$ on $n$ vertices with $\delta_{d}\left(G\right)\ge\left(\mu_{d}\left(k\right)+\gamma\right)\binom{n-d}{k-d}$,
there is a $K$-sparse $r$-absorber rooted on every $rk$-tuple of
vertices.
\end{lem}

As outlined at the beginning of this subsection, to enforce our local sparsity condition we will apply the Dirac threshold to subgraphs arising from a high-girth ``pattern''. This pattern will be constructed from the following bipartite graph.

\begin{lem}
\label{lem:high-girth-graph}
Fix $k$ and $K$. For
sufficiently large $n$, there is a $q$-regular bipartite graph $F$
with girth at least $K$ and at most $n$ edges, for some $q\ge\log^{2}n$
which is divisible by $k$.
\end{lem}

\begin{proof}There are many ways to prove this. For example, fix
a prime $p$ such that $n/2^{K+1}\le p^{K+1}\le n$ (which exists by Bertrand's
postulate), and consider the bipartite graph defined by Lazebnik and
Ustimenko in \cite{LU95}, which is $p$-regular, has $2p^{K}$ vertices
and has girth at least $K+5$. Then repeatedly delete perfect matchings
at most $k-1$ times until we arrive at a $q$-regular bipartite graph with $q$ divisible by $k$. (Note
that nonempty regular bipartite graphs always have perfect matchings). For large $n$, we have $q\ge p-\left(k-1\right)\ge\log^{2}n$.
\end{proof}

Another ingredient we will need is the following fact, which actually
already appeared in the proof of \cref{lem:random-subset-degrees}.

\begin{lem}
\label{lem:random-subset-degrees-2}
There is $c=c(k)>0$ (depending on $k$) such that the following holds. Consider an $n$-vertex
$k$-graph $G$ and consider a set $A$ of $d$ vertices with degree
at least $\left(\mu+\gamma\right)\binom{n-d}{k-d}$. Let $S$ be a
random subset of $Q\ge 2d$ vertices of $G$. Then with probability at least
$1-e^{-c\gamma^{2}Q}$, $A$ has degree at least
$\left(\mu+\gamma/2\right)\binom{Q-d}{k-d}$ into $S$.
\end{lem}

Now we are ready to prove \cref{lem:find-q-absorber}.

\begin{proof}[Proof of \cref{lem:find-q-absorber}] Consider
a graph $F$ as in \cref{lem:high-girth-graph}, and for a vertex
$v$, let $F\left(v\right)$ be the set of edges incident to $v$.
Then define a $q$-uniform hypergraph $L$ whose vertices are the
edges of $F$ and whose edges are the sets $F\left(v\right)$. Note
that the girth condition on $F$ transfers to $L$: the girth
of $L$ is at least $K$. Also, note that the two vertex parts of
$F$ correspond to two perfect matchings $M_{1}$ and $M_{2}$ partitioning the edges of $L$. Let $L'$ be obtained by deleting $rk$ vertices from one of
the edges of $M_{2}$, so that $L'$ has $rk$ vertices $z_{1},\dots,z_{rk}$
which have degree 1 in $L'$. Let $L''$ be the non-uniform hypergraph
obtained by deleting each $z_{i}$ from the edge it is contained in
(so that $L''$ has $rk$ edges with size only $q-1$, in addition to the edge of $L'$ with size only $q-rk$).

Now, consider an $rk$-tuple of vertices $\left(y_{1},\dots,y_{rk}\right)$ in $G$,
and consider a uniformly random injection $\phi:V\left(L''\right)\to V\left(G\right)\backslash\left\{ y_{1},\dots,y_{rk}\right\} $. Extend $\phi$ to
a map $V\left(L'\right)\to V\left(G\right)$ by taking $\phi\left(z_{i}\right)=y_{i}$
for each $i$.

Then, for each edge $e\in E\left(L'\right)$, note that $\phi\left(e\right)$
is ``almost'' a uniformly random subset of $q$ vertices of $G$.
To be precise, one can couple $\phi\left(e\right)$ with a uniformly random subset $S$ of $q=\Omega\left(\log^{2}n\right)$ vertices of $G$, in such a way that the size of the symmetric difference $|S\triangle \phi\left(e\right)|$ is at most $1+2rk$. By \cref{lem:random-subset-degrees-2} and
the union bound, with probability $1-o(n^{-k})$ every $d$-set of
vertices $U$ satisfies
\[
\deg_{S}\left(U\right)\ge\left(\mu_{d}\left(k\right)+\gamma/2\right)\binom{q-d}{k-d},
\]
implying that $\deg_{\phi\left(e\right)}\left(U\right)\ge\left(\mu_{d}\left(k\right)+\gamma/3\right)\binom{q-d}{k-d}$. By the union bound, a.a.s.\ this holds for each $e\in E\left(L'\right)$,
so fix such an outcome of $\phi$. Then for each $e\in E\left(L'\right)$, $G\left[\phi\left(e\right)\right]$
has minimum $d$-degree at least $\left(\mu_{d}\left(k\right)+\gamma/3\right)\binom{q-d}{k-d}$,
so has a perfect matching. The union of these perfect matchings gives
a $K$-sparse $r$-absorber rooted at $y_{1},\dots,y_{rk}$.\end{proof}

Now, we deduce \cref{lem:find-subabsorber} from \cref{lem:find-q-absorber}.

\begin{proof}[Proof of \cref{lem:find-subabsorber}]
Consider $x_{1},\dots,x_{k}$ as in
the theorem statement, and consider a random subset $U$ of $R$ vertices
of $G-\left\{ x_{1},\dots,x_{k}\right\} $, for some large $R$ to
be determined. Then by \cref{lem:random-subset-degrees-2} (with $d=1$) and \cref{lem:random-subset-degrees}, with probability at
least $1-\binom{R}{d}\left(\delta+e^{-\Omega\left(R\right)}\right)-ke^{-\Omega\left(R\right)}$
each $x_{i}$ has at least $\left(\eta/2\right)R$ neighbours in
$U$, and $G\left[U\right]$ has minimum $d$-degree at least $\left(\mu_{d}\left(k\right)+\eta/2\right)\binom{R-d}{k-d}$.
This probability is greater than zero for large $R$ and small $\delta>0$,
so we may fix such a choice of $U$.

For each $i$, choose an edge containing $x_{i}$ and $k$ vertices
in $U$, in such a way that these chosen edges form a matching $M$
(we can do this greedily). Let $y_{1},\dots,y_{\left(k-1\right)k}$
be the vertices in $V\left(M\right)\cap U$, and apply \cref{lem:find-q-absorber} to find
a $K$-sparse $\left(k-1\right)$-absorber $H$ rooted at $y_{1},\dots,y_{\left(k-1\right)k}$.
Then $M\cup H$ is a $K$-sparse absorber of order at most $R$ rooted
at $x_{1},\dots,x_{k}$.
\end{proof}

\section{Concluding remarks}
\label{sec:concluding}
We have proved that if $p\ge \max\{n^{-k/2+\gamma},C n^{-k+2}\}$, for any $\gamma>0$ and sufficiently large $C$, then the random $k$-graph $G\sim{\operatorname H}^{k}\left(n,p\right)$ typically obeys a relative version of any Dirac-type theorem for perfect matchings in hypergraphs. There are a number of compelling further directions of research.

It is natural to try to improve our assumption on $p$, with the eventual goal of removing it entirely (as in \cref{conj:resilience}). First, as mentioned in the introduction, we observe that the assumption $p\ge C n^{-k+2}$ can actually be weakened substantially (though this only affects the case $(d,k)=(1,3)$). The reason for this assumption was to ensure that all vertices have linear degree, so that an absorbing structure of linear size could be built greedily. The reason we needed an absorbing structure of linear size was that \cref{lem:almost-perfect} does not have effective bounds: it guarantees an almost-perfect matching covering all but $o(n)$ vertices, but since the regularity lemma is notorious for its extremely weak quantitative aspects, this $o(n)$ term is actually only very slightly sub-linear. However, it is possible to use a bootstrapping trick due to Nenadov and \v Skori\'c~\cite{NS20} to get a much stronger bound in the setting of \cref{lem:almost-perfect}, which allows us to make do with a much smaller absorbing structure. Using these ideas, it seems to be possible to take $p$ to be as small as about $n^{-4/3}$, in the case $(d,k)=(1,3)$. Actually, there is some hope of being able to remove the extra assumption $p\ge C n^{-k+2}$ altogether, by using the Aharoni--Haxell matching criterion (\cref{thm:hyper-match}) to build an absorbing structure, instead of building it greedily. We have not considered this in detail.

On the other hand, the assumption that $p$ is somewhat larger than $n^{-k/2}$ seems to be much more crucial. An absorber has at least $k/2$ times more edges than unrooted vertices, so absorbers of constant size simply will not exist for smaller $p$. We imagine that completely new ideas will be required to bypass this barrier.

Another interesting direction would be to consider spanning subgraphs other than perfect matchings. For example, a \emph{loose} cycle is a cyclically ordered collection of edges, such that only consecutive edges intersect, and then only in a single vertex. A \emph{tight} cycle is a cyclically ordered collection of vertices, such that every $k$ consecutive vertices form an edge. There is also a spectrum of different notions of cycles between these two extremes, and Dirac-type problems have been studied for Hamiltonian cycles of all these different types. We believe that it should be possible to adapt the methods in this paper to prove an analogue of \cref{thm:resilience} for loose Hamiltonian cycles, which are linear (no two edges intersect in more than one vertex) and behave in a very similar way to perfect matchings. It may also be possible to adapt our methods to study other types of Hamiltonian cycles, but this would probably require using different machinery from \cref{thm:KLR} (which only works for linear hypergraphs).

Finally, it may also be interesting to consider Dirac-type theorems relative to \emph{pseudorandom} hypergraphs, which are not random but satisfy some characteristic properties of random hypergraphs. Certain extremal problems relative to pseudorandom hypergraphs have been studied by Conlon, Fox and Zhao~\cite{CFZ15} in connection with the Green--Tao theorem on arithmetic progressions in the prime numbers, and the existence of perfect matchings in pseudorandom hypergraphs has been studied by H{\`a}n, Han and Morris~\cite{HHM20}. It seems plausible that the methods in this paper can be adapted to work for hypergraphs satisfying some notion of pseudorandomness, but we have not explored this further.

\medskip

\textbf{Acknowledgements.} We are grateful
to the referee for their extremely careful reading of the paper, and a large number of useful comments and suggestions.

\bibliographystyle{amsplain_initials_nobysame_nomr}

\bibliography{ref}

\providecommand{\bysame}{\leavevmode\hbox to3em{\hrulefill}\thinspace}
\providecommand{\MR}{\relax\ifhmode\unskip\space\fi MR }
\providecommand{\MRhref}[2]{%
  \href{http://www.ams.org/mathscinet-getitem?mr=#1}{#2}
}
\providecommand{\href}[2]{#2}
\begin{thebibliography}{10}

\bibitem{AH00}
R.~Aharoni and P.~Haxell, \emph{Hall's theorem for hypergraphs}, J. Graph
  Theory \textbf{35} (2000), no.~2, 83--88.

\bibitem{ABET20}
P.~Allen, J.~B\"{o}ttcher, J.~Ehrenm\"{u}ller, and A.~Taraz, \emph{The
  bandwidth theorem in sparse graphs}, Advances in Combinatorics (2020).

\bibitem{ABHKP16}
P.~Allen, J.~B{\"o}ttcher, H.~H{\`a}n, Y.~Kohayakawa, and Y.~Person,
  \emph{Blow-up lemmas for sparse graphs}, arXiv preprint arXiv:1612.00622
  (2016).

\bibitem{AFHRRS12}
N.~Alon, P.~Frankl, H.~Huang, V.~R\"{o}dl, A.~Ruci\'{n}ski, and B.~Sudakov,
  \emph{Large matchings in uniform hypergraphs and the conjecture of
  {E}rd{\H{o}}s and {S}amuels}, J. Combin. Theory Ser. A \textbf{119} (2012),
  no.~6, 1200--1215.

\bibitem{BCS11}
J.~Balogh, B.~Csaba, and W.~Samotij, \emph{Local resilience of almost spanning
  trees in random graphs}, Random Structures Algorithms \textbf{38} (2011),
  no.~1-2, 121--139.

\bibitem{BKS11}
S.~Ben-Shimon, M.~Krivelevich, and B.~Sudakov, \emph{On the resilience of
  {H}amiltonicity and optimal packing of {H}amilton cycles in random graphs},
  SIAM J. Discrete Math. \textbf{25} (2011), no.~3, 1176--1193.

\bibitem{BKT13}
J.~B\"{o}ttcher, Y.~Kohayakawa, and A.~Taraz, \emph{Almost spanning subgraphs
  of random graphs after adversarial edge removal}, Combin. Probab. Comput.
  \textbf{22} (2013), no.~5, 639--683.

\bibitem{CEP19}
D.~Clemens, J.~Ehrenm\"{u}ller, and Y.~Person, \emph{A {D}irac-type theorem for
  {B}erge cycles in random hypergraphs}, Electron. J. Combin. \textbf{27}
  (2020), no.~3, Paper No. 3.39, 23.

\bibitem{CG16}
D.~Conlon and W.~T. Gowers, \emph{Combinatorial theorems in sparse random
  sets}, Ann. of Math. (2) \textbf{184} (2016), no.~2, 367--454.

\bibitem{CGSS14}
D.~Conlon, W.~T. Gowers, W.~Samotij, and M.~Schacht, \emph{On the {{K}{\L}{R}}
  conjecture in random graphs}, Israel J. Math. \textbf{203} (2014), no.~1,
  535--580.

\bibitem{Con14}
D.~Conlon, \emph{Combinatorial theorems relative to a random set}, Proceedings
  of the {I}nternational {C}ongress of {M}athematicians---{S}eoul 2014. {V}ol.
  {IV}, Kyung Moon Sa, Seoul, 2014, pp.~303--327.

\bibitem{CFZ15}
D.~Conlon, J.~Fox, and Y.~Zhao, \emph{A relative {S}zemer\'{e}di theorem},
  Geom. Funct. Anal. \textbf{25} (2015), no.~3, 733--762.

\bibitem{Dir52}
G.~A. Dirac, \emph{Some theorems on abstract graphs}, Proc. London Math. Soc.
  (3) \textbf{2} (1952), 69--81.

\bibitem{EGP91}
P.~Erd\H{o}s, A.~Gy\'{a}rf\'{a}s, and L.~Pyber, \emph{Vertex coverings by
  monochromatic cycles and trees}, J. Combin. Theory Ser. B \textbf{51} (1991),
  no.~1, 90--95.

\bibitem{ES66}
P.~Erd\H{o}s and M.~Simonovits, \emph{A limit theorem in graph theory}, Studia
  Sci. Math. Hungar. \textbf{1} (1966), 51--57.

\bibitem{ES46}
P.~Erd\"{o}s and A.~H. Stone, \emph{On the structure of linear graphs}, Bull.
  Amer. Math. Soc. \textbf{52} (1946), 1087--1091.

\bibitem{FH19}
A.~Ferber and L.~Hirschfeld, \emph{Co-degrees resilience for perfect matchings
  in random hypergraphs}, Electron. J. Combin. \textbf{27} (2020), no.~1, Paper
  No. 1.40.

\bibitem{FK19}
A.~Ferber and M.~Kwan, \emph{Almost all {S}teiner triple systems are almost
  resolvable}, Forum Math. Sigma \textbf{8} (2020), Paper No. e39, 24.

\bibitem{FNNP17}
A.~Ferber, R.~Nenadov, A.~Noever, U.~Peter, and N.~\v{S}kori\'{c}, \emph{Robust
  {H}amiltonicity of random directed graphs}, J. Combin. Theory Ser. B
  \textbf{126} (2017), 1--23.

\bibitem{FK18}
P.~Frankl and A.~Kupavskii, \emph{The {E}rd{\H o}s matching conjecture and
  concentration inequalities}, arXiv preprint arXiv:1806.08855 (2018).

\bibitem{FKNP19}
K.~Frankston, J.~Kahn, B.~Narayanan, and J.~Park, \emph{Thresholds versus
  fractional expectation-thresholds}, Ann. of Math. (2) \textbf{194} (2021),
  no.~2, 475--495.

\bibitem{FK08}
A.~Frieze and M.~Krivelevich, \emph{On two {H}amilton cycle problems in random
  graphs}, Israel J. Math. \textbf{166} (2008), 221--234.

\bibitem{GKLMO19}
S.~Glock, D.~K\"{u}hn, A.~Lo, R.~Montgomery, and D.~Osthus, \emph{On the
  decomposition threshold of a given graph}, J. Combin. Theory Ser. B
  \textbf{139} (2019), 47--127.

\bibitem{Gow07}
W.~T. Gowers, \emph{Hypergraph regularity and the multidimensional
  {S}zemer\'{e}di theorem}, Ann. of Math. (2) \textbf{166} (2007), no.~3,
  897--946.

\bibitem{GIKM17}
C.~Greenhill, M.~Isaev, M.~Kwan, and B.~D. McKay, \emph{The average number of
  spanning trees in sparse graphs with given degrees}, European J. Combin.
  \textbf{63} (2017), 6--25.

\bibitem{HPS09}
H.~H\`an, Y.~Person, and M.~Schacht, \emph{On perfect matchings in uniform
  hypergraphs with large minimum vertex degree}, SIAM J. Discrete Math.
  \textbf{23} (2009), no.~2, 732--748.

\bibitem{HHM20}
H.~H{\`a}n, J.~Han, and P.~Morris, \emph{Factors and loose {H}amilton cycles in
  sparse pseudo-random hypergraphs}, Proceedings of the Fourteenth Annual
  ACM-SIAM Symposium on Discrete Algorithms, SIAM, 2020, pp.~702--717.

\bibitem{HLS12}
H.~Huang, C.~Lee, and B.~Sudakov, \emph{Bandwidth theorem for random graphs},
  J. Combin. Theory Ser. B \textbf{102} (2012), no.~1, 14--37.

\bibitem{JLR00}
S.~Janson, T.~{\L{}}uczak, and A.~Ruci{\'n}ski, \emph{Random graphs}, Cambridge
  University Press, 2000.

\bibitem{JKV08}
A.~Johansson, J.~Kahn, and V.~Vu, \emph{Factors in random graphs}, Random
  Structures Algorithms \textbf{33} (2008), no.~1, 1--28.

\bibitem{Kah19}
J.~Kahn, \emph{Asymptotics for {S}hamir's problem}, arXiv preprint
  arXiv:1909.06834 (2019).

\bibitem{Kha16}
I.~Khan, \emph{Perfect matchings in 4-uniform hypergraphs}, J. Combin. Theory
  Ser. B \textbf{116} (2016), 333--366.

\bibitem{KLR97}
Y.~Kohayakawa, T.~{\L}uczak, and V.~R\"{o}dl, \emph{On {$K^4$}-free subgraphs
  of random graphs}, Combinatorica \textbf{17} (1997), no.~2, 173--213.

\bibitem{KR03}
Y.~Kohayakawa and V.~R\"{o}dl, \emph{Szemer\'{e}di's regularity lemma and
  quasi-randomness}, Recent advances in algorithms and combinatorics, CMS Books
  Math./Ouvrages Math. SMC, vol.~11, Springer, New York, 2003, pp.~289--351.

\bibitem{KNRS10}
Y.~Kohayakawa, B.~Nagle, V.~R\"{o}dl, and M.~Schacht, \emph{Weak hypergraph
  regularity and linear hypergraphs}, J. Combin. Theory Ser. B \textbf{100}
  (2010), no.~2, 151--160.

\bibitem{Kri97}
M.~Krivelevich, \emph{Triangle factors in random graphs}, Combin. Probab.
  Comput. \textbf{6} (1997), no.~3, 337--347.

\bibitem{Kwa16}
M.~Kwan, \emph{Almost all {S}teiner triple systems have perfect matchings},
  Proc. Lond. Math. Soc. (3) \textbf{121} (2020), no.~6, 1468--1495.

\bibitem{LU95}
F.~Lazebnik and V.~A. Ustimenko, \emph{Explicit construction of graphs with an
  arbitrary large girth and of large size}, vol.~60, 1995, ARIDAM VI and VII
  (New Brunswick, NJ, 1991/1992), pp.~275--284.

\bibitem{LS12}
C.~Lee and B.~Sudakov, \emph{Dirac's theorem for random graphs}, Random
  Structures Algorithms \textbf{41} (2012), no.~3, 293--305.

\bibitem{LM15}
A.~Lo and K.~Markstr\"{o}m, \emph{{$F$}-factors in hypergraphs via absorption},
  Graphs Combin. \textbf{31} (2015), no.~3, 679--712.

\bibitem{Man07}
W.~Mantel, \emph{Problem 28 (solution by {H}.~{G}ouwentak, {W}.~{M}antel,
  {J}.~{T}eixeira de {M}attes, {F}.~{S}chuh, and {W}.~{A}.~{W}ythoff)},
  Wiskundige Opgaven (1907), no.~10, 60--61.

\bibitem{Mon14}
R.~Montgomery, \emph{Embedding bounded degree spanning trees in random graphs},
  arXiv preprint arXiv:1405.6559 (2014).

\bibitem{Mon17}
R.~Montgomery, \emph{{H}amiltonicity in random graphs is born resilient}, J.
  Combin. Theory Ser. B \textbf{139} (2019), 316--341.

\bibitem{Mon19}
R.~Montgomery, \emph{Spanning trees in random graphs}, Adv. Math. \textbf{356}
  (2019), 106793, 92.

\bibitem{Mon20}
R.~Montgomery, \emph{Hamiltonicity in random directed graphs is born
  resilient}, Combin. Probab. Comput. \textbf{29} (2020), no.~6, 900--942.

\bibitem{NST19}
R.~Nenadov, A.~Steger, and M.~Truji\'{c}, \emph{Resilience of perfect matchings
  and {H}amiltonicity in random graph processes}, Random Structures Algorithms
  \textbf{54} (2019), no.~4, 797--819.

\bibitem{NS20}
R.~Nenadov and N.~\v{S}kori\'{c}, \emph{On {K}oml\'{o}s' tiling theorem in
  random graphs}, Combin. Probab. Comput. \textbf{29} (2020), no.~1, 113--127.

\bibitem{NS17}
A.~Noever and A.~Steger, \emph{Local resilience for squares of almost spanning
  cycles in sparse random graphs}, Electron. J. Combin. \textbf{24} (2017),
  no.~4, Paper No. 4.8.

\bibitem{Pik08}
O.~Pikhurko, \emph{Perfect matchings and {$K^3_4$}-tilings in hypergraphs of
  large codegree}, Graphs Combin. \textbf{24} (2008), no.~4, 391--404.

\bibitem{RR10}
V.~R\"{o}dl and A.~Ruci\'{n}ski, \emph{Dirac-type questions for hypergraphs---a
  survey (or more problems for {E}ndre to solve)}, An irregular mind, Bolyai
  Soc. Math. Stud., vol.~21, J\'{a}nos Bolyai Math. Soc., Budapest, 2010,
  pp.~561--590.

\bibitem{RRS06}
V.~R\"{o}dl, A.~Ruci\'{n}ski, and E.~Szemer\'{e}di, \emph{A {D}irac-type
  theorem for 3-uniform hypergraphs}, Combin. Probab. Comput. \textbf{15}
  (2006), no.~1-2, 229--251.

\bibitem{RRS09}
V.~R\"{o}dl, A.~Ruci\'{n}ski, and E.~Szemer\'{e}di, \emph{Perfect matchings in
  large uniform hypergraphs with large minimum collective degree}, J. Combin.
  Theory Ser. A \textbf{116} (2009), no.~3, 613--636.

\bibitem{RS04}
V.~R\"{o}dl and J.~Skokan, \emph{Regularity lemma for {$k$}-uniform
  hypergraphs}, Random Structures Algorithms \textbf{25} (2004), no.~1, 1--42.

\bibitem{Sch16}
M.~Schacht, \emph{Extremal results for random discrete structures}, Ann. of
  Math. (2) \textbf{184} (2016), no.~2, 333--365.

\bibitem{SV08}
B.~Sudakov and V.~H. Vu, \emph{Local resilience of graphs}, Random Structures
  Algorithms \textbf{33} (2008), no.~4, 409--433.

\bibitem{TV06}
T.~Tao and V.~Vu, \emph{Additive combinatorics}, Cambridge Studies in Advanced
  Mathematics, vol. 105, Cambridge University Press, Cambridge, 2006.

\bibitem{SST18}
N.~\v{S}kori\'{c}, A.~Steger, and M.~Truji\'{c}, \emph{Local resilience of an
  almost spanning {$k$}-cycle in random graphs}, Random Structures Algorithms
  \textbf{53} (2018), no.~4, 728--751.

\bibitem{Zha16}
Y.~Zhao, \emph{Recent advances on {D}irac-type problems for hypergraphs},
  Recent trends in combinatorics, IMA Vol. Math. Appl., vol. 159, Springer,
  [Cham], 2016, pp.~145--165.

\end{thebibliography}

\end{document}